\documentclass[preprint,12pt]{elsarticle}%
\usepackage{lineno,hyperref}
\usepackage{amsmath}
\usepackage{amsfonts}
\usepackage{amssymb}
\usepackage{graphicx}%
\providecommand{\U}[1]{\protect \rule{.1in}{.1in}}
\modulolinenumbers[5]
\pdfoutput=1

\newtheorem{theorem}{Theorem}[section]

\newtheorem{definition}[theorem]{Definition}

\newtheorem{example}[theorem]{Example}

\newtheorem{lemma}[theorem]{Lemma}

\newtheorem{proposition}[theorem]{Proposition}
\newtheorem{remark}[theorem]{Remark}

\newenvironment{proof}[1][Proof]{\noindent \textbf{#1.} }{\  \rule{0.5em}{0.5em}}
\numberwithin{equation}{section}
\newtheorem{sch}{Scheme}
\begin{document}
\begin{frontmatter}
\title{An efficient numerical method for forward-backward stochastic differential
equations driven by $G$-Brownian motion \tnoteref{mytitlenote}}
\tnotetext[mytitlenote]{This work was funded by National Key R\&D Program of China (No. 2018YFA0703900)
and the National Natural Science Foundation of China (No. 11671231).}
%
\author[mymainaddress,mysecondaryaddress]{Mingshang Hu}
\ead{humingshang@sdu.edu.cn}
\author[mysecondaryaddress]{Lianzi Jiang\corref{mycorrespondingauthor}}
\cortext[mycorrespondingauthor]{Corresponding author}
\ead{jianglianzi95@163.com}
\address[mymainaddress]{Zhongtai Securities Institute for Financial Studies,
Shandong University, Jinan, Shandong 250100, PR China}
\address[mysecondaryaddress]{Zhongtai Securities Institute for Financial Studies,
Shandong University, Jinan, Shandong 250100, PR China}
\begin{abstract}
In this paper, we study the numerical method for solving
forward-backward stochastic differential equations driven by $G$-Brownian
motion ($G$-FBSDEs) which correspond to fully nonlinear partial differential
equations (PDEs). First, we give an approximate conditional $G$-expectation
and obtain some feasible methods to calculate the distribution of $G$-Brownian
motion. On this basis, some efficient numerical schemes for $G$-FBSDEs are
then proposed. We rigorously analyze the errors of the proposed schemes and
prove the convergence. Finally, several numerical experiments are presented to
demonstrate the accuracy of our schemes.
\end{abstract}
\end{frontmatter}

\linenumbers

\section{Introduction}

Considering the volatility uncertainty in the financial market, Avellaneda et
al. \cite{ALP1995}, Lyons \cite{L1995}, and Dellacherie \cite{D1972} initially
studied the superhedging of European options. But for the superhedging of
general path-dependence options, their methods are no longer suitable.
Recently, Peng \cite{P2004,P2007,P2008,P2010} introduced the $G$-expectation
theory to deal with this problem, see also \cite{DM2006,STZh2012} for
different approaches. Under the $G$-expectation framework, a new notion of
$G$-normal distribution was introduced, which is the limit distribution
corresponding to the central limit theorem. The notion of $G$-normal
distribution plays the same important role in the theory of sublinear
expectation as that of normal distribution in the classical probability
theory. Based on it, a new type of $G$-Brownian motion and the related
stochastic calculus of It\^{o}'s type have been established.

In this paper, we study the feasible numerical scheme for the following
forward-backward stochastic differential equation driven by $G$-Brownian
motion ($G$-FBSDE for short, we always use Einstein convention):%
\begin{equation}
\left \{
\begin{aligned} dX_{t} & =b\left( t,X_{t}\right) dt+h_{ij}\left( t,X_{t}\right) d\langle B^{i},B^{j}\rangle_{t}+\sigma \left( t,X_{t}\right) dB_{t}, \;X_{0}=x_{0}\in \mathbb{R}^{m},\\ -dY_{t} & =f\left( t,X_{t},Y_{t},Z_{t}\right) dt+g_{ij}\left( t,X_{t},Y_{t},Z_{t}\right) d\langle B^{i},B^{j}\rangle_{t}-Z_{t}dB_{t}-dK_{t},\\ \text{ \ }Y_{T} & =\phi \left( X_{T}\right) , \end{aligned}\right.
\label{1.0}%
\end{equation}
where $B_{t}=(B_{t}^{1},\ldots,B_{t}^{d})^{\top}$ is a $d$-dimensional
$G$-Brownian motion defined in the $G$-expectation space $(\Omega_{T}
,L_{G}^{1}(\Omega_{T}),\mathbb{\hat{E}})$, $K$ is a decreasing $G$-martingale,
$b,h_{ij}:[0,T]\times \mathbb{R}^{m}\rightarrow \mathbb{R}^{m}$, $\sigma
:[0,T]\times \mathbb{R}^{m}\rightarrow \mathbb{R}^{m\times d}$, $\phi
:\mathbb{R}^{m}\rightarrow \mathbb{R}$, $f,g_{ij}:[0,T]\times \mathbb{R}%
^{m}\times \mathbb{R\times R}^{d}\rightarrow \mathbb{R}$, and $h_{ij}=h_{ji}$,
$g_{ij}=g_{ji},$ for $1\leq i,j\leq d$. The first equation in (\ref{1.0}) is
the stochastic differential equation driven by $G$-Brownian motion ($G$-SDE),
and the second equation is the backward stochastic differential equation
driven by $G$-Brownian motion ($G$-BSDE).

Under some standard conditions on $b,h_{ij},\sigma,f,g_{ij}$, and $\phi$, Hu
et al. \cite{HJPS20141} proved the existence and uniqueness of the solution
for $G$-BSDEs. Moreover, they \cite{HJPS20142} studied the nonlinear
Feynman-Kac formula under the $G$-framework (see Theorem
\ref{Nonlinear Feynman-Kac}), which built the relationship between $G$-FBSDEs
and fully nonlinear partial differential equations (PDEs). In addition,
Cheridito et al. \cite{CSTV2007} and Soner et al. \cite{STZh2012} developed a
new type of fully nonlinear FBSDEs, called the 2FBSDEs, which is also
associated with a class of fully nonlinear PDEs. Tremendous efforts have been
made on the numerical computation of FBSDEs or 2FBSDEs (see, e.g.,
\cite{B1997,BD2007,C2013,CR2015,CM2014,EHJ2017,FTW2011,GLW2005,GT2016,GZhZh2015,
KZhZh2015,MZ2005,MT2006,RO2016,TLG2020,Z2004,ZZ2002,ZhCP2006,ZhFZh2014} and
the references therein), but little seems to be known about the numerical
results for $G$-FBSDEs. There are some obstacles in developing the numerical
scheme for $G$-FBSDEs. On the one hand, there is no density representation for
the $G$-normal distribution, and the classical numerical integration methods
are no longer applicable. On the other hand, owing to the sublinear nature of
$G$-expectation, many conclusions under the linear expectation can not be
extended to the $G$-expectation, which increases the difficulty of theoretical analysis.

For the first obstacle, inspired by the $G$-expectation representation (see
Theorem \ref{E repesent}), we give an approximate conditional $G$-expectation,
which leads to the feasible methods to calculate the distribution of
$G$-Brownian motion, including the trinomial tree rule and the Gauss-Hermite
quadrature rule (see also \cite{D2012,YZh2016} for other different methods).
On this basis, some efficient numerical schemes for solving $G$-FBSDEs are
proposed. For the second obstacle, using the property of $G$-expectation, we
rigorously analyze the errors of the proposed schemes and prove the
convergence. Some examples are given to numerically demonstrate the accuracy
of the proposed schemes. To the best of our knowledge, this is the first
attempt to design the numerical scheme for $G$-FBSDEs.

The paper is organized as follow. In Section 2, we recall some preliminaries
used in the $G$-framework. We propose the numerical schemes for $G$-FBSDEs in
Section 3. The convergence result of the proposed schemes is rigorously proved
in Section 4. In Section 5, we extend our result to the case of
multi-dimensional Brownian motion. Finally, various numerical examples
illustrate the performance of our schemes in Section 6.

\section{Preliminaries}

For any fixed $T>0$, let $\Omega_{T}=C_{0}([0,T];\mathbb{R}^{d})$ be the space
of $\mathbb{R}^{d}$-valued continuous paths on $[0,T]$ with $\omega_{0}=0$,
endowed with the supremum norm, and $B_{t}(\omega)=\omega_{t}$ be the
canonical process. Set
\[
Lip(\Omega_{T})=\{ \varphi(B_{t_{1}},\ldots,B_{t_{n}}):n\geq1,t_{1}%
,\ldots,t_{n}\in \lbrack0,T],\varphi \in C_{b,Lip}(\mathbb{R}^{d\times
n})\} \text{,}%
\]
where $C_{b,Lip}(\mathbb{R}^{d\times n})$ denotes the set of bounded Lipschitz
functions on $\mathbb{R}^{d\times n}$.

Peng \cite{P2007} constructed a consistent sublinear expectation space
$(\Omega_{T},Lip(\Omega_{T}),\mathbb{\hat{E}},(\mathbb{\hat{E}}_{t})_{t\geq
0})$, called a $G$-expectation space, and the canonical process $(B_{t}%
)_{t\geq0}$ is called a $G$-Brownian motion. The monotonic and sublinear
function $G:\mathbb{S}(d)\rightarrow \mathbb{R}$ is defined by
\[
G\left(  A\right)  :=\frac{1}{2}\mathbb{\hat{E}[}\left \langle AB_{1}%
,B_{1}\right \rangle ],\text{ \ }A\in \mathbb{S}(d),
\]
where $\mathbb{S}(d)$ denotes the collection of $d\times d$ symmetric
matrices. Note that there exists a bounded and closed subset $\Sigma \subset$
$\mathbb{S}^{+}(d)$ such that
\[
G\left(  A\right)  =\frac{1}{2}\sup_{Q\in \Sigma}tr[QA],\text{ \ }%
A\in \mathbb{S}(d),
\]
where $\mathbb{S}^{+}(d)\ $denotes the collection of nonnegative definite
elements in $\mathbb{S}(d)$. In this paper, we assume that $G$ is
non-degenerate, i.e., there exist some constants $0<\underline{\sigma}^{2}
\leq \bar{\sigma}^{2}<\infty$ such that $\frac{1}{2}\underline{\sigma}
^{2}tr[A-B]\leq G(A)-G(B) \leq \frac{1}{2}\bar{\sigma}^{2}tr[A-B]$ for
$A\geq B$.

For each given $p\geq1$, define $\Vert X\Vert_{L_{G}^{p}}=(\mathbb{\hat{E}%
}[|X|^{p}])^{1/p}$ for $X\in Lip(\Omega_{T})$, and denote by $L_{G}^{p}%
(\Omega_{T})$ the completion of $Lip(\Omega_{T})$ under the norm $\Vert
\cdot \Vert_{L_{G}^{p}}$. Then for $t\in \left[  0,T\right]  $, $\mathbb{\hat
{E}}_{t}[\cdot]$ can be extended continuously to $L_{G}^{1}(\Omega_{T})$.

\begin{definition}
Let $M_{G}^{0}(0,T)$ be the collection of processes in the following form: for
a given partition $\{t_{0},t_{1},\ldots,t_{N}\}$ of $[0,T]$,
\[
\eta_{t}(\omega)=\sum_{n=0}^{N-1}\xi_{n}(\omega)I_{[t_{n},t_{n+1})}(t),
\]
where $\xi_{n}\in Lip(\Omega_{t_{n}})$, $n=0,1,\ldots,N-1$.
\end{definition}

For each given $p\geq1$, define $\Vert \eta \Vert_{M_{G}^{p}}=(\mathbb{\hat{E}%
}[\int_{0}^{T}\left \vert \eta_{s}\right \vert ^{p}ds])^{1/p}$ for $\eta \in
M_{G}^{0}(0,T)$, and denote by $M_{G}^{p}(0,T)$ the completion of $M_{G}%
^{0}(0,T)$ under $\Vert \cdot \Vert_{M_{G}^{p}}$. Denote by $\langle
B\rangle_{t}:=(\langle B^{i},B^{j}\rangle_{t})_{i,j=1}^{d}$\ the mutual
variation process. For $\xi_{t}^{i}\in$ $M_{G}^{2}(0,T)$ and $\eta_{t}^{ij}%
\in$ $M_{G}^{1}(0,T)$, $1\leq i,j\leq d$, the $G$-It\^{o} integral $\int
_{0}^{T}\xi_{t}dB_{t}:=\sum \limits_{i=1}^{d}\int_{0}^{T}\xi_{t}^{i}dB_{t}^{i}%
$\ and $\int_{0}^{T}\eta_{t}d\langle B\rangle_{t}:=\sum \limits_{i,j=1}^{d}%
\int_{0}^{T}\eta_{t}^{ij}d\langle B^{i},B^{j}\rangle_{t}$ are well defined.
The readers may refer to \cite{P2007,P2008,P2010}\ for more details of
$G$-It\^{o}'s integral and $G$-FBSDEs.

The following result is the representation theorem for the $G$-expectation.

\begin{theorem}
[\cite{DHP2011,HP2009}]\label{E repesent}There exists a weakly compact family
$\mathcal{P}$ of probability measures on $\left(  \Omega_{T},\mathcal{B}(\Omega_{T})\right)$ such that
\[
\mathbb{\hat{E}}\left[  X\right]  =\sup_{P\in \mathcal{P}}\mathbb{E}_{P}\left[
X\right]  ,\text{ for all }X\in L_{G}^{1}(\Omega_{T}).
\]
$\mathcal{P}$ is called a set that represents $\mathbb{\hat{E}}$.
\end{theorem}

Let $(X_{s}^{t,x},Y_{s}^{t,x},Z_{s}^{t,x},K_{s}^{t,x})$, for $s\in \lbrack
t,T],$ be the solution of $(\ref{1.0})$ starting from $t$ with $X_{t}=x$, that
is,
\begin{align}
X_{s}^{t,x}  &  =x+\int_{t}^{s}b\left(  r,X_{r}^{t,x}\right)  dr+\int_{t}%
^{s}h_{ij}\left(  r,X_{r}^{t,x}\right)  d\langle B^{i},B^{j}\rangle_{r}%
+\int_{t}^{s}\sigma \left(  r,X_{r}^{t,x}\right)  dB_{r},\label{G-SDE}\\
Y_{s}^{t,x}  &  =\phi \left(  X_{T}^{t,x}\right)  +\int_{s}^{T}f\left(
r,X_{r}^{t,x},Y_{r}^{t,x},Z_{r}^{t,x}\right)  dr+\int_{s}^{T}g_{ij}\left(
r,X_{r}^{t,x},Y_{r}^{t,x},Z_{r}^{t,x}\right)  d\langle B^{i},B^{j}\rangle
_{r}\label{G-BSDE}\\
&  \text{ \  \ }-\int_{s}^{T}Z_{r}^{t,x}dB_{r}-\left(  K_{T}^{t,x}-K_{s}%
^{t,x}\right)  \text{.}\nonumber
\end{align}

Next, we shall present the nonlinear Feynman-Kac formula under the $G$-framework.

\begin{theorem}
[\cite{HJPS20142}]\label{Nonlinear Feynman-Kac}Assume that the functions
$b,h_{ij},\sigma,f,g_{ij}$, and $\phi$ are uniformly Lipschitz continuous with
respect to $(x,y,z)$ and continuous with respect to $t$, and $h_{ij}=h_{ji}$,
$g_{ij}=g_{ji}$ for $1\leq i,j\leq d$. Let $u(t,x):=Y_{t}^{t,x}$ for
$(t,x) \in \lbrack0,T]\times \mathbb{R}^{m}$.\ Then $u(t,x)$ is the unique
solution of the following PDE:
\begin{equation}
\left \{
\begin{array}
[c]{l}%
\partial_{t}u+F\left(  D_{x}^{2}u,D_{x}u,u,x,t\right)  =0,\\
u\left(  T,x\right)  =\phi \left(  x\right)  ,
\end{array}
\right.  \label{F-K}%
\end{equation}
where%
\begin{align*}
&  F\left(  D_{x}^{2}u,D_{x}u,u,x,t\right)  =G\left(  H\left(  D_{x}%
^{2}u,D_{x}u,u,x,t\right)  \right)  +\left \langle b\left(  t,x\right)
,D_{x}u\right \rangle \\
&  \quad \quad \quad \quad \quad \quad \quad \quad \quad \;\;+f\left(
t,x,u,\left \langle \sigma_{1}\left(  t,x\right)  ,D_{x}u\right \rangle
,\ldots,\left \langle \sigma_{d}\left(  t,x\right)  ,D_{x}u\right \rangle
\right)  ,\\
&  H_{ij}\left(  D_{x}^{2}u,D_{x}u,u,x,t\right)  =\langle D_{x}^{2}u\sigma
_{i}\left(  t,x\right)  ,\sigma_{j}\left(  t,x\right)  \rangle+2\langle
D_{x}u,h_{ij}\left(  t,x\right)  \rangle \\
&  \quad \quad \quad \quad \quad \quad \quad \quad \quad \quad \;\;+2g_{ij}\left(
t,x,u,\langle \sigma_{1}\left(  t,x\right)  ,D_{x}u\rangle,\ldots,\langle
\sigma_{d}\left(  t,x\right)  ,D_{x}u\rangle \right)  .
\end{align*}

\end{theorem}

\begin{remark}
\label{F-K represent}By Theorem \ref{Nonlinear Feynman-Kac}, the solution of
the $G$-FBSDE (\ref{1.0}) can be represented as%
\[%
\begin{array}
[c]{l}%
\displaystyle Y_{t}=u(t,X_{t}),\text{ \ }Z_{t}=D_{x}u(t,X_{t})^{\top}%
\sigma(t,X_{t}),\\
\displaystyle K_{t}=\frac{1}{2}\int_{0}^{t}H_{ij}\left(  D_{x}^{2}%
u(s,X_{s}),D_{x}u(s,X_{s}),u(s,X_{s}),X_{s},s\right)  d\langle B^{i}%
,B^{j}\rangle_{s}\\
\displaystyle \text{ \  \  \  \  \ }-\int_{0}^{t}G\left(  H\left(  D_{x}%
^{2}u(s,X_{s}),D_{x}u(s,X_{s}),u(s,X_{s}),X_{s},s\right)  \right)  ds.
\end{array}
\]

\end{remark}

For readers' convenience, we list the main notations of this paper as follows.

\begin{itemize}
\item $\Delta B_{n+1}:=B_{t_{n+1}}-B_{t_{n}}$,\ $\  \Delta \langle
B\rangle_{n+1}:=\langle B\rangle_{t_{n+1}}-\langle B\rangle_{t_{n}}$;

\item $\mathbb{\hat{E}}_{t}^{x}[\eta]:=\mathbb{\hat{E}}[\eta|X_{t}=x]$, the
conditional $G$-expectation of the random variable $\eta$;

\item $\mathbb{E}_{t}^{\sigma,x}[\eta]:=\mathbb{E}_{P^{\sigma}}[\eta|X_{t}%
=x]$, the conditional expectation of the random variable $\eta$ under the
probability measure $P^{\sigma}\in \mathcal{P}$;

\item $\mathbb{\tilde{E}}_{t}^{x}[\eta]:=\sup \limits_{\sigma \in \{
\underline{\sigma},\overline{\sigma}\}}\mathbb{E}_{t}^{\sigma,x}[\eta]$, the
approximate conditional $G$-expectation of the random variable $\eta$;

\item $\mathbb{\tilde{E}}^{TR}[\varphi(B_{t},\langle B\rangle_{t}%
)]:=\sup \limits_{\sigma \in \{ \underline{\sigma},\overline{\sigma}\}}%
\sum \limits_{i=1}^{3}\omega_{i}^{\sigma}\varphi(\sqrt{t}q_{i},tq_{i}^{2})$,
the trinomial tree rule approximation for the distribution of $\varphi
(B_{t},\langle B\rangle_{t})$ with the weights $\{ \omega_{i}^{\sigma}%
\}_{i=1}^{3}$ and the nodes $\{q_{i}\}_{i=1}^{3}$.
\end{itemize}

\section{Numerical schemes for $G$-FBSDEs}

We first consider the one-dimensional $G$-Brownian motion case. The results
for the multi-dimensional $G$-Brownian motion case will be given in Section
\ref{multi case}. For the time interval $\left[  0,T\right]  $, we introduce a
uniform time partition $0=t_{0}<t_{1}<\cdots<t_{N}=T$ with $\Delta
t=t_{n+1}-t_{n}=T/N$. We use the following Euler scheme to approximate the
$G$-SDE (\ref{G-SDE}):
\begin{equation}
X^{n+1}=X^{n}+b(t_{n},X^{n})\Delta t+h(t_{n},X^{n})\Delta \langle
B\rangle_{n+1}+\sigma(t_{n},X^{n})\Delta B_{n+1}\text{,} \label{X^n}%
\end{equation}
and $X^{0}=x_{0}$, for $n=0,\ldots,N-1$, where $\Delta B_{n+1}\sim N(0,\Delta
t\Sigma)$ and $\Delta \langle B\rangle_{n+1}\sim N(\Delta t\Sigma,0)$ with
$\Sigma=[\underline{\sigma}^{2},\overline{\sigma}^{2}]$.

Let $(X_{t}^{t_{n},X^{n}},Y_{t}^{t_{n},X^{n}},Z_{t}^{t_{n},X^{n}},K_{t}
^{t_{n},X^{n}})$, for $t\in \left[  t_{n},T\right]  $, be the solution of
(\ref{G-SDE})-(\ref{G-BSDE}) with $(t,x)=(t_{n},X^{n})$, and denote
$f_{t}^{t_{n},X^{n}}=f(t,X_{t}^{t_{n},X^{n} },Y_{t}^{t_{n},X^{n}},Z_{t}%
^{t_{n},X^{n}})$ and $g_{t}^{t_{n},X^{n}} =g(t,X_{t}^{t_{n},X^{n}},$
$Y_{t}^{t_{n},X^{n}},Z_{t}^{t_{n},X^{n}})$. Then, for $n=0,1,\ldots,N-1,$
\begin{align}
X_{t_{n+1}}^{t_{n},X^{n}}  &  =X^{n}+\int_{t_{n}}^{t_{n+1}}b(t,X_{t}%
^{t_{n},X^{n}})dt+\int_{t_{n}}^{t_{n+1}}h(t,X_{t}^{t_{n},X^{n}})d\langle
B\rangle_{t}+\int_{t_{n}}^{t_{n+1}}\sigma(t,X_{t}^{t_{n},X^{n}})dB_{t}%
,\nonumber \\
Y_{t_{n}}^{t_{n},X^{n}}  &  =Y_{t_{n+1}}^{t_{n},X^{n}}+\int_{t_{n}}^{t_{n+1}%
}f_{t}^{t_{n},X^{n}}dt+\int_{t_{n}}^{t_{n+1}}g_{t}^{t_{n},X^{n}}d\langle
B\rangle_{t}-\int_{t_{n}}^{t_{n+1}}Z_{t}^{t_{n},X^{n}}dB_{t}-(K_{t_{n+1}%
}^{t_{n},X^{n}}-K_{t_{n}}^{t_{n},X^{n}})\text{.} \label{2.2}%
\end{align}

\subsection{Conditional $G$-expectation approximation}

Taking the conditional $G$-expectation$\  \mathbb{\hat{E}}_{t_{n}}^{X^{n}%
}\left[  \cdot \right]  :=\mathbb{\hat{E}}\left[  \cdot|X_{t_{n}}=X^{n}\right]
$ on (\ref{2.2}) and noting that $K_{t}$ is a $G$-martingale,\ we have
\begin{equation}
Y_{t_{n}}^{t_{n},X^{n}}=\mathbb{\hat{E}}_{t_{n}}^{X^{n}}\left[  Y_{t_{n+1}%
}^{t_{n},X^{n}}+\int_{t_{n}}^{t_{n+1}}f_{t}^{t_{n},X^{n}}dt+\int_{t_{n}%
}^{t_{n+1}}g_{t}^{t_{n},X^{n}}d\langle B\rangle_{t}\right]  . \label{2.3}%
\end{equation}
Using the right rectangle formula in (\ref{2.3}) and approximating the forward
process, we obtain
\begin{align}
Y_{t_{n}}^{t_{n},X^{n}}  &  =\mathbb{\hat{E}}_{t_{n}}^{X^{n}}\left[
Y_{t_{n+1}}^{t_{n},X^{n}}+f_{t_{n+1}}^{t_{n},X^{n}}\Delta t+g_{t_{n+1}}%
^{t_{n},X^{n}}\Delta \langle B\rangle_{n+1}\right]  +\hat{R}_{y_{1}}^{n}\\
&  =\mathbb{\hat{E}}_{t_{n}}^{X^{n}}\left[  Y_{t_{n+1}}^{t_{n+1},X^{n+1}%
}+f_{t_{n+1}}^{t_{n+1},X^{n+1}}\Delta t+g_{t_{n+1}}^{t_{n+1},X^{n+1}}%
\Delta \langle B\rangle_{n+1}\right]  +\hat{R}_{y_{1}}^{n}+\hat{R}_{y_{2}}%
^{n},\nonumber
\end{align}
where
\begin{align}
&  \hat{R}_{y_{1}}^{n}=\mathbb{\hat{E}}_{t_{n}}^{X^{n}}\left[  Y_{t_{n+1}%
}^{t_{n},X^{n}}+%
{\textstyle \int_{t_{n}}^{t_{n+1}}}
f_{t}^{t_{n},X^{n}}dt+%
{\textstyle \int_{t_{n}}^{t_{n+1}}}
g_{t}^{t_{n},X^{n}}d\langle B\rangle_{t}\right]  \quad \quad \quad \label{2.5}\\
&  \text{ \  \  \  \  \  \  \ }-\mathbb{\hat{E}}_{t_{n}}^{X^{n}}\left[  Y_{t_{n+1}%
}^{t_{n},X^{n}}+f_{t_{n+1}}^{t_{n},X^{n}}\Delta t+g_{t_{n+1}}^{t_{n},X^{n}%
}\Delta \langle B\rangle_{n+1}\right]  ,\nonumber
\end{align}
and%
\begin{align}
&  \hat{R}_{y_{2}}^{n}=\mathbb{\hat{E}}_{t_{n}}^{X^{n}}\left[  Y_{t_{n+1}%
}^{t_{n},X^{n}}+f_{t_{n+1}}^{t_{n},X^{n}}\Delta t+g_{t_{n+1}}^{t_{n},X^{n}%
}\Delta \langle B\rangle_{n+1}\right] \label{2.6}\\
&  \text{ \  \  \  \  \  \  \ }-\mathbb{\hat{E}}_{t_{n}}^{X^{n}}\left[  Y_{t_{n+1}%
}^{t_{n+1},X^{n+1}}+f_{t_{n+1}}^{t_{n+1},X^{n+1}}\Delta t+g_{t_{n+1}}%
^{t_{n+1},X^{n+1}}\Delta \langle B\rangle_{n+1}\right]  .\nonumber
\end{align}
For any $\eta \in L_{G}^{1}(\Omega_{T})$, define the approximate conditional
$G$-expectation $\mathbb{\tilde{E}}_{t_{n}}^{X^{n}}[\eta]$ by
\begin{equation}
\mathbb{\tilde{E}}_{t_{n}}^{X^{n}}[\eta]:=\sup \limits_{\sigma \in \left \{
\underline{\sigma},\overline{\sigma}\right \}  }\mathbb{E}_{t_{n}}%
^{\sigma,X^{n}}\left[  \eta \right]  :=\sup_{\sigma \in \left \{  \underline
{\sigma},\overline{\sigma}\right \}  }\mathbb{E}_{P^{\sigma}}\left[  \left.
\eta \right \vert X_{t_{n}}=X^{n}\right]  , \label{E^wan}%
\end{equation}
where $P^{\sigma}\in \mathcal{P}$ is a probability measure,
under which $(B_{t})_{t\geq0}$ is the classical Brownian motion with
$\mathbb{E}_{P^{\sigma}}[B_{t}]=0$\ and $\mathbb{E}_{P^{\sigma}}[B_{t}%
^{2}]=\sigma^{2}t$. Then we have
\begin{equation}
Y_{t_{n}}^{t_{n},X^{n}}=\mathbb{\tilde{E}}_{t_{n}}^{X^{n}}\left[  Y_{t_{n+1}%
}^{t_{n+1},X^{n+1}}+f_{t_{n+1}}^{t_{n+1},X^{n+1}}\Delta t+g_{t_{n+1}}%
^{t_{n+1},X^{n+1}}\Delta \langle B\rangle_{n+1}\right]  +\tilde{R}_{y}^{n},
\label{2.7}%
\end{equation}
where $\tilde{R}_{y}^{n}=\hat{R}_{y_{1}}^{n}+\hat{R}_{y_{2}}^{n}+\tilde
{R}_{y_{3}}^{n}$, with
\begin{align}
\tilde{R}_{y_{3}}^{n}  &  =\mathbb{\hat{E}}_{t_{n}}^{X^{n}}\left[  Y_{t_{n+1}%
}^{t_{n+1},X^{n+1}}+f_{t_{n+1}}^{t_{n+1},X^{n+1}}\Delta t+g_{t_{n+1}}%
^{t_{n+1},X^{n+1}}\Delta \langle B\rangle_{n+1}\right] \label{2.8}\\
&  \text{ \  \ }-\mathbb{\tilde{E}}_{t_{n}}^{X^{n}}\left[  Y_{t_{n+1}}%
^{t_{n+1},X^{n+1}}+f_{t_{n+1}}^{t_{n+1},X^{n+1}}\Delta t+g_{t_{n+1}}%
^{t_{n+1},X^{n+1}}\Delta \langle B\rangle_{n+1}\right]  .\nonumber
\end{align}

Now let us multiply both sides of (\ref{2.2}) by $\Delta B_{n+1}$ and take the
conditional expectation $\mathbb{E}_{t_{n}}^{\sigma,X^{n}}\left[
\cdot \right]  $, for any $\sigma \in \left \{  \underline{\sigma},\overline
{\sigma}\right \}  $,\ then we obtain
\begin{equation}
\mathbb{E}_{t_{n}}^{\sigma,X^{n}}\left[  \int_{t_{n}}^{t_{n+1}}Z_{t}%
^{t_{n},X^{n}}dB_{t}\Delta B_{n+1}\right]  =\mathbb{E}_{t_{n}}^{\sigma,X^{n}%
}\left[  Y_{t_{n+1}}^{t_{n+1},X^{n+1}}\Delta B_{n+1}\right]  +R_{z_{1}%
}^{n,\sigma}+R_{z_{2}}^{n,\sigma}, \label{2.9}%
\end{equation}
where%
\begin{align}
R_{z_{1}}^{n,\sigma}  &  =\mathbb{E}_{t_{n}}^{\sigma,X^{n}}\left[
{\textstyle \int_{t_{n}}^{t_{n+1}}}
f_{t}^{t_{n},X^{n}}dt\Delta B_{n+1}+%
{\textstyle \int \nolimits_{t_{n}}^{t_{n+1}}}
g_{t}^{t_{n},X^{n}}d\langle B\rangle_{t}\Delta B_{n+1}\right] \label{2.10}\\
&  \text{ \  \ }-\mathbb{E}_{t_{n}}^{\sigma,X^{n}}\left[  (K_{t_{n+1}}%
^{t_{n},X^{n}}-K_{t_{n}}^{t_{n},X^{n}})\Delta B_{n+1}\right]  ,\nonumber \\
R_{z_{2}}^{n,\sigma}  &  =\mathbb{E}_{t_{n}}^{\sigma,X^{n}}\left[  Y_{t_{n+1}%
}^{t_{n},X^{n}}\Delta B_{n+1}\right]  -\mathbb{E}_{t_{n}}^{\sigma,X^{n}%
}\left[  Y_{t_{n+1}}^{t_{n+1},X^{n+1}}\Delta B_{n+1}\right]  . \label{2.11}%
\end{align}
Noting that $B_{t}\sim N(0,\sigma^{2}t)$ under $\mathbb{E}^{\sigma}$, by the
It\^{o} isometry formula, we have%
\begin{equation}
\mathbb{E}_{t_{n}}^{\sigma,X^{n}}\left[  \int_{t_{n}}^{t_{n+1}}Z_{t}%
^{t_{n},X^{n}}dB_{t}\Delta B_{n+1}\right]  =\sigma^{2}\mathbb{E}_{t_{n}%
}^{\sigma,X^{n}}\left[  \int_{t_{n}}^{t_{n+1}}Z_{t}^{t_{n},X^{n}}dt\right]  .
\label{2.12}%
\end{equation}
Combining (\ref{2.9}) and (\ref{2.12}), for $\sigma \in \{ \underline{\sigma
},\overline{\sigma}\}$, we have
\begin{equation}
\sigma^{2}\Delta tZ_{t_{n}}^{t_{n},X^{n}}=\mathbb{E}_{t_{n}}^{\sigma,X^{n}%
}\left[  Y_{t_{n+1}}^{t_{n+1},X^{n+1}}\Delta B_{n+1}\right]  +\tilde{R}%
_{z}^{n,\sigma}\text{,} \label{2.13}%
\end{equation}
where $\tilde{R}_{z}^{n,\sigma}=R_{z_{1}}^{n,\sigma}+R_{z_{2}}^{n,\sigma
}-R_{z_{3}}^{n,\sigma}$ and
\begin{equation}
R_{z_{3}}^{n,\sigma}=\sigma^{2}\mathbb{E}_{t_{n}}^{\sigma,X^{n}}\left[
{\textstyle \int_{t_{n}}^{t_{n+1}}}
Z_{t}^{t_{n},X^{n}}dt\right]  -\sigma^{2}\Delta tZ_{t_{n}}^{t_{n},X^{n}}.
\label{2.14}%
\end{equation}
Thus we obtain
\begin{align}
&  Y_{t_{n}}^{t_{n},X^{n}}=\mathbb{\tilde{E}}_{t_{n}}^{X^{n}}\left[
Y_{t_{n+1}}^{t_{n+1},X^{n+1}}+f_{t_{n+1}}^{t_{n+1},X^{n+1}}\Delta
t+g_{t_{n+1}}^{t_{n+1},X^{n+1}}\Delta \langle B\rangle_{n+1}\right]  +\tilde
{R}_{y}^{n},\label{ref1}\\
&  \sigma^{2}\Delta tZ_{t_{n}}^{t_{n},X^{n}}=\mathbb{E}_{t_{n}}^{\sigma,X^{n}
}\left[  Y_{t_{n+1}}^{t_{n+1},X^{n+1}}\Delta B_{n+1}\right]  +\tilde{R}
_{z}^{n,\sigma}, \label{ref2}%
\end{align}
for any $\sigma \in \{  \underline{\sigma},\overline{\sigma}\}  $.

\subsection{Trinomial tree rule\label{Subection Appro G}}

Notice that $Y_{t_{n+1}}^{t_{n+1},X^{n+1}}$, $f_{t_{n+1}}^{t_{n+1},X^{n+1}}$,
$g_{t_{n+1}}^{t_{n+1},X^{n+1}}\Delta \langle B\rangle_{n+1}$, and $Y_{t_{n+1}
}^{t_{n+1},X^{n+1}}\Delta B_{n+1}$ in (\ref{ref1})-(\ref{ref2}) are the
functions of $X^{n+1}$, $\Delta B_{n+1}$, and $\Delta \langle B\rangle_{n+1}$.
For a function $\varphi:\mathbb{R}^{m}\times \mathbb{R}\times \mathbb{R}%
\rightarrow \mathbb{R}$, denote $\varphi_{t_{n+1}}=\varphi(X^{n+1},\Delta
B_{n+1},$ $\Delta \langle B\rangle_{n+1})$, we define the following trinomial
tree rule to approximate $\mathbb{\tilde{E}}_{t_{n}}^{X^{n}}[\varphi_{t_{n+1}%
}]$:
\begin{equation}
\mathbb{\tilde{E}}_{t_{n}}^{TR,X^{n}}\left[  \varphi_{t_{n+1}}\right]
:=\sup \limits_{\sigma \in \left \{  \underline{\sigma},\overline{\sigma}\right \}
}\mathbb{E}_{t_{n}}^{TR,\sigma,X^{n}}\left[  \varphi_{t_{n+1}}\right]
:=\sup \limits_{\sigma \in \left \{  \underline{\sigma},\overline{\sigma}\right \}
}\sum_{i=1}^{3}\omega_{i}^{\sigma}\tilde{\varphi}\left(  X^{n},\sqrt{\Delta
t}q_{i},\Delta tq_{i}^{2}\right)  , \label{E TR}%
\end{equation}
where $\tilde{\varphi}\left(  x,y,z\right)  =\varphi \left(  x+b(t_{n},x)\Delta
t+\sigma(t_{n},x)y+h(t_{n},x)z,y,z\right)  $ and
\begin{equation}
\left \{
\begin{array}
[c]{ll}%
q_{1}=-1, & \omega_{1}^{\sigma}=\sigma^{2}/2;\\
q_{2}=0, & \omega_{2}^{\sigma}=1-\sigma^{2};\\
q_{3}=1, & \omega_{3}^{\sigma}=\sigma^{2}/2.
\end{array}
\right.  \label{qi}%
\end{equation}
More precisely, we define the associated discrete sublinear expectation
\[
\mathbb{\tilde{E}}^{TR}\left[  \varphi_{t_{n+1}}\right]  =\mathbb{\tilde{E}%
}_{t_{0}}^{TR,x_{0}}\left[  \mathbb{\tilde{E}}_{t_{1}}^{TR,X^{1}}\left[
\cdots \mathbb{\tilde{E}}_{t_{n}}^{TR,X^{n}}[\varphi_{t_{n+1}}]\right]
\right]  .
\]

\begin{remark}
\label{remark2}From the definition of $\{ \omega_{i}^{\sigma}\}_{i=1}^{3}$ and
$\{q_{i}\}_{i=1}^{3}$, one can check that
\[
\sum \limits_{i=1}^{3}\omega_{i}^{\sigma}=1,\text{ \ }\sum \limits_{i=1}%
^{3}\omega_{i}^{\sigma}q_{i}=0,\text{ \ }\sum \limits_{i=1}^{3}\omega
_{i}^{\sigma}q_{i}^{2}=\sigma^{2}.
\]

\end{remark}

It is easy to verify the following properties.

\begin{proposition}
\label{PropE}Assume that $\varphi_{t_{n+1}}=\varphi \left(  X^{n+1},\Delta
B_{n+1},\Delta \langle B\rangle_{n+1}\right)  .$ Then

\begin{enumerate}
\item[$\left(  a\right)  $] $\mathbb{\tilde{E}}^{TR}\left[  \mathbb{\tilde{E}%
}_{t_{n}}^{TR,X^{n}}\left[  \varphi_{t_{n+1}}\right]  \right]  =\mathbb{\tilde
{E}}^{TR}\left[  \varphi_{t_{n+1}}\right]  ;$

\item[$\left(  b\right)  $] $\left \vert \mathbb{E}_{t_{n}}^{TR,\sigma,X^{n}%
}\left[  \varphi_{t_{n+1}}\right]  \right \vert ^{2}\leq \mathbb{E}_{t_{n}%
}^{TR,\sigma,X^{n}}\left[  |\varphi_{t_{n+1}}|^{2}\right]  ;$

\item[$\left(  c\right)  $] $\left \vert \mathbb{E}_{t_{n}}^{TR,\sigma,X^{n}%
}\left[  \varphi_{t_{n+1}}\Delta B_{n+1}\right]  \right \vert ^{2}\leq \left(
\mathbb{E}_{t_{n}}^{TR,\sigma,X^{n}}\left[  |\varphi_{t_{n+1}}|^{2}\right]
-\left \vert \mathbb{E}_{t_{n}}^{TR,\sigma,X^{n}}\left[  \varphi_{t_{n+1}%
}\right]  \right \vert ^{2}\right)  \sigma^{2}\Delta t.$
\end{enumerate}
\end{proposition}

From the definitions of $\mathbb{\tilde{E}}_{t_{n}}^{TR,X^{n}}\left[
\cdot \right]  $ and $\mathbb{E}_{t_{n}}^{TR,\sigma,X^{n}}\left[  \cdot \right]
$, we have the following approximations:%
\begin{align}
&  \mathbb{\tilde{E}}_{t_{n}}^{X^{n}}\left[  Y_{t_{n+1}}^{t_{n+1},X^{n+1}%
}+f_{t_{n+1}}^{t_{n+1},X^{n+1}}\Delta t+g_{t_{n+1}}^{t_{n+1},X^{n+1}}%
\Delta \langle B\rangle_{n+1}\right] \label{3.9}\\
&  =\mathbb{\tilde{E}}_{t_{n}}^{TR,X^{n}}\left[  Y_{t_{n+1}}^{t_{n+1},X^{n+1}%
}+f_{t_{n+1}}^{t_{n+1},X^{n+1}}\Delta t+g_{t_{n+1}}^{t_{n+1},X^{n+1}}%
\Delta \langle B\rangle_{n+1}\right]  +R_{y}^{T,n},\nonumber
\end{align}
and%
\begin{equation}
\mathbb{E}_{t_{n}}^{\sigma,X^{n}}\left[  Y_{t_{n+1}}^{t_{n+1},X^{n+1}}\Delta
B_{n+1}\right]  =\mathbb{E}_{t_{n}}^{TR,\sigma,X^{n}}\left[  Y_{t_{n+1}%
}^{t_{n+1},X^{n+1}}\Delta B_{n+1}\right]  +R_{z}^{T,n,\sigma}, \label{3.10}%
\end{equation}
where the errors%
\begin{align}
&  R_{y}^{T,n}=\mathbb{\tilde{E}}_{t_{n}}^{X^{n}}\left[  Y_{t_{n+1}}%
^{t_{n+1},X^{n+1}}+f_{t_{n+1}}^{t_{n+1},X^{n+1}}\Delta t+g_{t_{n+1}}%
^{t_{n+1},X^{n+1}}\Delta \langle B\rangle_{n+1}\right] \label{3.7}\\
&  \text{ \  \  \  \  \  \  \  \ }-\mathbb{\tilde{E}}_{t_{n}}^{TR,X^{n}}\left[
Y_{t_{n+1}}^{t_{n+1},X^{n+1}}+f_{t_{n+1}}^{t_{n+1},X^{n+1}}\Delta
t+g_{t_{n+1}}^{t_{n+1},X^{n+1}}\Delta \langle B\rangle_{n+1}\right]
,\nonumber \\
&  R_{z}^{T,n,\sigma}=\mathbb{E}_{t_{n}}^{\sigma,X^{n}}\left[  Y_{t_{n+1}%
}^{t_{n+1},X^{n+1}}\Delta B_{n+1}\right]  -\mathbb{E}_{t_{n}}^{TR,\sigma
,X^{n}}\left[  Y_{t_{n+1}}^{t_{n+1},X^{n+1}}\Delta B_{n+1}\right]  .
\label{3.8}%
\end{align}
Based on (\ref{ref1}), (\ref{ref2}), (\ref{3.9}),$\ $and (\ref{3.10}), for
$\sigma \in \left \{  \underline{\sigma},\overline{\sigma}\right \}  $, we obtain
the following reference equations
\begin{align}
&  Y_{t_{n}}^{t_{n},X^{n}}=\mathbb{\tilde{E}}_{t_{n}}^{TR,X^{n}}\left[
Y_{t_{n+1}}^{t_{n+1},X^{n+1}}+f_{t_{n+1}}^{t_{n+1},X^{n+1}}\Delta
t+g_{t_{n+1}}^{t_{n+1},X^{n+1}}\Delta \langle B\rangle_{n+1}\right]  +\tilde
{R}_{y}^{n}+R_{y}^{T,n},\label{3.5}\\
&  \sigma^{2}\Delta tZ_{t_{n}}^{t_{n},X^{n}}=\mathbb{E}_{t_{n}}^{TR,\sigma
,X^{n}}\left[  Y_{t_{n+1}}^{t_{n+1},X^{n+1}}\Delta B_{n+1}\right]  +\tilde
{R}_{z}^{n,\sigma}+R_{z}^{T,n,\sigma}. \label{3.6}%
\end{align}

\subsection{The discrete scheme}

Let $Y^{n}$ and $Z^{n}$ be the numerical approximations for the solutions
$Y_{t}$ and $Z_{t}$ of the $G$-FBSDE (\ref{1.0}) at time $t_{n}$,
respectively, and denote $f^{n+1}=f(t_{n+1},X^{n+1},Y^{n+1},Z^{n+1})$,
$g^{n+1}=g(t_{n+1},X^{n+1},Y^{n+1},Z^{n+1})$. Based on the reference equations
(\ref{3.5})-(\ref{3.6}), we propose the following numerical scheme for solving
the $G$-FBSDE (\ref{1.0}):

\begin{sch}
\label{scheme 1}Given random variables $Y^{N}$ and $Z^{N}$, for $n=N-1,\ldots
,0$ and $\sigma \in \{ \underline{\sigma},\overline{\sigma}\}$, solve random
variables $Y^{n}=Y^{n}(X^{n}) $ and $Z^{n}=Z^{n}(X^{n})$ from
\begin{align}
Y^{n}  &  =\mathbb{\tilde{E}}_{t_{n}}^{TR,X^{n}}\left[  Y^{n+1}+f^{n+1}\Delta
t+g^{n+1}\Delta \langle B\rangle_{n+1}\right]  ,\label{yn_tr}\\
Z^{n}  &  =\mathbb{E}_{t_{n}}^{TR,\sigma,X^{n}}\left[  Y^{n+1}\Delta
B_{n+1}\right]  /\sigma^{2}\Delta t, \label{zn_tr}%
\end{align}
with%
\begin{equation}
X^{n+1}=X^{n}+b(t_{n},X^{n})\Delta t+h(t_{n},X^{n})\Delta \langle
B\rangle_{n+1}+\sigma(t_{n},X^{n})\Delta B_{n+1}. \label{xn_Tri}%
\end{equation}

\end{sch}

Noticed that when solving for $Z^{n}$, $\mathbb{E}_{t_{n}}^{TR,\sigma,X^{n}%
}[\cdot]$ in (\ref{zn_tr}) can be selected arbitrarily among $\sigma \in \{
\underline{\sigma},\overline{\sigma}\}$. Besides, the parameter $\sigma$ that
$\mathbb{\tilde{E}}_{t_{n}}^{TR,X^{n}}[\cdot]$ in (\ref{yn_tr})\ reaches its
maximum can be obtained in the numerical test, denoted as $\sigma^{n}$. So one
natural choice for solving $Z^{n}$ is to let $\sigma=$ $\sigma^{n}$. In this
case, Scheme \ref{scheme 1} becomes

\begin{sch}
\label{trinomial tree scheme}Given random variables $Y^{N}$ and $Z^{N}$, for
$n=N-1,\ldots,0$, solve random variables $Y^{n}=Y^{n}(X^{n}) $ and
$Z^{n}=Z^{n}(X^{n}) $ from
\begin{align}
Y^{n}  &  =\mathbb{E}_{t_{n}}^{TR,\sigma^{n},X^{n}}\left[  Y^{n+1}%
+f^{n+1}\Delta t+g^{n+1}\Delta \langle B\rangle_{n+1}\right]  ,\label{Yn_Tri}\\
Z^{n}  &  =\mathbb{E}_{t_{n}}^{TR,\sigma^{n},X^{n}}\left[  Y^{n+1}\Delta
B_{n+1}\right]  /(\sigma^{n})^{2}\Delta t, \label{Zn_Tri}%
\end{align}
with%
\begin{equation}
X^{n+1}=X^{n}+b(t_{n},X^{n})\Delta t+h(t_{n},X^{n})\Delta \langle
B\rangle_{n+1}+\sigma(t_{n},X^{n})\Delta B_{n+1}. \label{Xn_Tri}%
\end{equation}

\end{sch}

In fact, Scheme \ref{scheme 1} is independent of the parameter $\sigma$ when
solving $Z^{n}$ (see Lemma \ref{lemma1}). Therefore, Scheme \ref{scheme 1} and
Scheme \ref{trinomial tree scheme} are essentially the same. In order to
simplify the presentation, in the following, we will only perform the
convergence analysis for Scheme \ref{trinomial tree scheme}.

\section{Convergence analysis}

In this section, we focus on the convergence analysis of our discrete schemes.
In the sequel, $C$ represents a generic constant which does not depend on the
time partition and may be different from line to line.

For $n=0,1,\ldots,N-1$, assume that $\mathbb{\tilde{E}}_{t_{n}}^{TR,X^{n}}$ in
(\ref{3.5}) reaches its maximum at $\mathbb{E}_{t_{n}}^{TR,\tilde{\sigma}
^{n},X^{n}}$ with the parameter $\tilde{\sigma}^{n}$ and let $\sigma
=\tilde{\sigma}^{n}$ in (\ref{3.6}), that is,%
\begin{align}
&  Y_{t_{n}}^{t_{n},X^{n}}=\mathbb{E}_{t_{n}}^{TR,\tilde{\sigma}^{n},X^{n}%
}\left[  Y_{t_{n+1}}^{t_{n+1},X^{n+1}}+f_{t_{n+1}}^{t_{n+1},X^{n+1}}\Delta
t+g_{t_{n+1}}^{t_{n+1},X^{n+1}}\Delta \langle B\rangle_{n+1}\right]  +\tilde
{R}_{y}^{n}+R_{y}^{G,n},\label{4.1}\\
&  (\tilde{\sigma}^{n})^{2}\Delta tZ_{t_{n}}^{t_{n},X^{n}}=\mathbb{E}_{t_{n}%
}^{TR,\tilde{\sigma}^{n},X^{n}}\left[  Y_{t_{n+1}}^{t_{n+1},X^{n+1}}\Delta
B_{n+1}\right]  +\tilde{R}_{z}^{n,\tilde{\sigma}^{n}}+R_{z}^{G,n,\tilde
{\sigma}^{n}}. \label{4.10}%
\end{align}
Define the related process $\tilde{Z}^{n}$ as follows
\begin{equation}
\tilde{Z}^{n}=\mathbb{E}_{t_{n}}^{TR,\tilde{\sigma}^{n},X^{n}}\left[
Y^{n+1}\Delta B_{n+1}\right]  /(\tilde{\sigma}^{n})^{2}\Delta t. \label{Z^wan}%
\end{equation}

\subsection{A useful theorem}

We now present an important theorem that will be useful in our convergence analysis.

\begin{theorem}
\label{theorem 1} Let $(X_{t}^{t_{n},X^{n}},Y_{t}^{t_{n},X^{n}},Z_{t}%
^{t_{n},X^{n}})_{t_{n}\leq t\leq T}$ and $(X^{n},Y^{n},Z^{n})(n=0,1,\ldots,N)$
be the solutions of \eqref{G-SDE}-\eqref{G-BSDE} and Scheme
\ref{trinomial tree scheme}, respectively. Assume that the functions $f$ and
$g$ are Lipschitz continuous with respect to $(x,y,z)$. Then, for sufficiently
small $\Delta t$ and $n=0,1,\ldots,N-1$,
\begin{align}
&  \mathbb{\tilde{E}}^{TR}\left[  |Y_{t_{n}}^{t_{n},X^{n}}-Y^{n}|^{2}+C\Delta
t|Z_{t_{n}}^{t_{n},X^{n}}-Z^{n}|^{2}\right] \label{theorem1}\\
&  \leq e^{CT}\mathbb{\tilde{E}}^{TR}\left[  |Y_{t_{N}}^{t_{N},X^{N}}%
-Y^{N}|^{2}+C\Delta t|Z_{t_{N}}^{t_{N},X^{N}}-Z^{N}|^{2}\right] \nonumber \\
&  \text{ \  \ }+C\Delta t\sum \limits_{i=n}^{N-1}\mathbb{\tilde{E}}^{TR}\left[
|\tilde{Z}^{i}-Z^{i}|^{2}\right]  +\frac{C}{\Delta t}\sum \limits_{i=n}%
^{N-1}\mathbb{\tilde{E}}^{TR}\mathbb{[}|\tilde{R}_{y}^{i}|^{2}+|R_{y}%
^{G,i}|^{2}+|\tilde{R}_{z}^{i}|^{2}+|R_{z}^{G,i}|^{2}],\  \nonumber
\end{align}
where $\tilde{Z}^{n}$ is defined in \eqref{Z^wan},
$\tilde{R}_{z}^{n}=\tilde{R}_{z}^{n,\sigma^{n}}\vee \tilde{R}_{z}
^{n,\tilde{\sigma}^{n}}$, $R_{z}^{T,n}=R_{z}^{T,n,\sigma^{n}}\vee
R_{z}^{T,n,\tilde{\sigma}^{n}}$, and $\tilde{R}_{y}^{n},R_{y}^{T,n},\tilde
{R}_{z}^{n,\sigma}$, $R_{z}^{T,n,\sigma}$ are defined in \eqref{2.7}, \eqref{3.7}, \eqref{2.13}
and \eqref{3.8}, respectively.
\end{theorem}

\begin{proof}
For simplicity, we set
\begin{align*}
&  \delta Y_{n}=Y_{t_{n}}^{t_{n},X^{n}}-Y^{n}\text{, \  \  \  \ }\delta
Z_{n}=Z_{t_{n}}^{t_{n},X^{n}}-Z^{n},\\
&  \delta f_{n}=f(t_{n},X^{n},Y_{t_{n}}^{t_{n},X^{n}},Z_{t_{n}}^{t_{n},X^{n}%
})-f(t_{n},X^{n},Y^{n},Z^{n}),\\
&  \delta g_{n}=g(t_{n},X^{n},Y_{t_{n}}^{t_{n},X^{n}},Z_{t_{n}}^{t_{n},X^{n}%
})-g(t_{n},X^{n},Y^{n},Z^{n}).
\end{align*}
Step 1. We consider the case of $Y^{n}>Y_{t_{n}}^{t_{n},X^{n}}$. Subtracting
(\ref{4.1}) from (\ref{Yn_Tri}), we have%
\begin{align}
Y^{n}-Y_{t_{n}}^{t_{n},X^{n}}  &  =\mathbb{E}_{t_{n}}^{TR,\sigma^{n},X^{n}%
}\left[  Y^{n+1}+f^{n+1}\Delta t+g^{n+1}\Delta \langle B\rangle_{n+1}\right]
-\tilde{R}_{y}^{n}-R_{y}^{T,n}\label{4.2}\\
&  \text{ \  \ }-\mathbb{E}_{t_{n}}^{TR,\tilde{\sigma}^{n},X^{n}}\left[
Y_{t_{n+1}}^{t_{n+1},X^{n+1}}+f_{t_{n+1}}^{t_{n+1},X^{n+1}}\Delta
t+g_{t_{n+1}}^{t_{n+1},X^{n+1}}\Delta \langle B\rangle_{n+1}\right]  .\nonumber
\end{align}
Noting that%
\begin{align}
&  \mathbb{E}_{t_{n}}^{TR,\sigma^{n},X^{n}}\left[  Y_{t_{n+1}}^{t_{n+1}%
,X^{n+1}}+f_{t_{n+1}}^{t_{n+1},X^{n+1}}\Delta t+g_{t_{n+1}}^{t_{n+1},X^{n+1}%
}\Delta \langle B\rangle_{n+1}\right] \label{comp}\\
&  \leq \mathbb{E}_{t_{n}}^{TR,\tilde{\sigma}^{n},X^{n}}\left[  Y_{t_{n+1}%
}^{t_{n+1},X^{n+1}}+f_{t_{n+1}}^{t_{n+1},X^{n+1}}\Delta t+g_{t_{n+1}}%
^{t_{n+1},X^{n+1}}\Delta \langle B\rangle_{n+1}\right]  ,\nonumber
\end{align}
which implies%
\begin{equation}
0<-\delta Y_{n}\leq-\mathbb{E}_{t_{n}}^{TR,\sigma^{n},X^{n}}\left[  \delta
Y_{n+1}+\delta f_{n+1}\Delta t+\delta g_{n+1}\Delta \langle B\rangle
_{n+1}\right]  -\tilde{R}_{y}^{n}-R_{y}^{T,n}.
\end{equation}
By the Lipschitz continuity of $f$, $g$ and the definition of $\mathbb{E}%
_{t_{n}}^{TR,\sigma^{n},X^{n}}$, we have%
\begin{equation}
\left \vert \delta Y_{n}\right \vert \leq \left \vert \mathbb{E}_{t_{n}%
}^{TR,\sigma^{n},X^{n}}\left[  \delta Y_{n+1}\right]  \right \vert +C\Delta
t\mathbb{E}_{t_{n}}^{TR,\sigma^{n},X^{n}}\left[  \left \vert \delta
Y_{n+1}\right \vert +\left \vert \delta Z_{n+1}\right \vert \right]  +|\tilde
{R}_{y}^{n}|+|R_{y}^{T,n}|. \label{4.3}%
\end{equation}
Taking square on both side of (\ref{4.3}) and using the inequality
$(a+b)^{2}\leq(1+\gamma \Delta t)a^{2}+(1+\frac{1}{\gamma \Delta t})b^{2}$,
$\gamma>0$, we deduce%
\begin{align}
\left \vert \delta Y_{n}\right \vert ^{2}  &  \leq(1+\gamma \Delta t)\left \vert
\mathbb{E}_{t_{n}}^{TR,\sigma^{n},X^{n}}\left[  \delta Y_{n+1}\right]
\right \vert ^{2}\label{4.4}\\
&  \text{ \  \  \ }+C(1+\frac{1}{\gamma \Delta t})\left(  \Delta t\right)
^{2}\mathbb{E}_{t_{n}}^{TR,\sigma^{n},X^{n}}\left[  \left \vert \delta
Y_{n+1}\right \vert ^{2}+\left \vert \delta Z_{n+1}\right \vert ^{2}\right]
\nonumber \\
&  \text{ \  \  \ }+C(1+\frac{1}{\gamma \Delta t})\left(  |\tilde{R}_{y}^{n}%
|^{2}+|R_{y}^{T,n}|^{2}\right)  .\nonumber
\end{align}
Let $\sigma=\sigma^{n}$ in (\ref{3.6}). Combining this with (\ref{Zn_Tri}), we
obtain%
\begin{equation}
(\sigma^{n})^{2}\delta Z_{n}\Delta t=\mathbb{E}_{t_{n}}^{TR,\sigma^{n},X^{n}%
}\left[  \delta Y_{n+1}\Delta B_{n+1}\right]  +\tilde{R}_{z}^{n,\sigma^{n}%
}+R_{z}^{T,n,\sigma^{n}}. \label{4.5}%
\end{equation}
Taking the square of both sides of (\ref{4.5}), by Proposition \ref{PropE}, we
can deduce%
\begin{align}
\frac{(\sigma^{n})^{2}}{C}\Delta t\left \vert \delta Z_{n}\right \vert ^{2}  &
\leq \mathbb{E}_{t_{n}}^{TR,\sigma^{n},X^{n}}\left[  |\delta Y_{n+1}%
|^{2}\right]  -\left \vert \mathbb{E}_{t_{n}}^{TR,\sigma^{n},X^{n}}\left[
\delta Y_{n+1}\right]  \right \vert ^{2}\label{4.6}\\
&  \text{ \  \  \ }+\frac{C}{\Delta t}\left(  |\tilde{R}_{z}^{n,\sigma^{n}}%
|^{2}+|R_{z}^{T,n,\sigma^{n}}|^{2}\right)  .\nonumber
\end{align}
Putting together (\ref{4.4}) and (\ref{4.6}), we get
\begin{align*}
\left \vert \delta Y_{n}\right \vert ^{2}+\frac{\underline{\sigma}^{2}}{C}\Delta
t\left \vert \delta Z_{n}\right \vert ^{2}  &  \leq(1+\gamma \Delta
t)\mathbb{E}_{t_{n}}^{TR,\sigma^{n},X^{n}}\left[  |\delta Y_{n+1}|^{2}\right]
\\
&  \text{ \  \  \ }+(1+\gamma \Delta t)\frac{C\Delta t}{\gamma}\mathbb{E}_{t_{n}%
}^{TR,\sigma^{n},X^{n}}\left[  \left \vert \delta Y_{n+1}\right \vert
^{2}+\left \vert \delta Z_{n+1}\right \vert ^{2}\right] \\
&  \text{ \  \  \ }+\frac{C}{\Delta t}\left(  |\tilde{R}_{y}^{n}|^{2}%
+|R_{y}^{T,n}|^{2}+|\tilde{R}_{z}^{n,\sigma^{n}}|^{2}+|R_{z}^{T,n,\sigma^{n}%
}|^{2}\right) \\
&  \leq(1+C\Delta t)\mathbb{E}_{t_{n}}^{TR,\sigma^{n},X^{n}}\left[  |\delta
Y_{n+1}|^{2}+\frac{\underline{\sigma}^{2}}{C}\Delta t\left \vert \delta
Z_{n+1}\right \vert ^{2}\right] \\
\text{ \  \  \  \  \ }  &  \text{\  \  \  \ }+\frac{C}{\Delta t}\left(  |\tilde
{R}_{y}^{n}|^{2}+|R_{y}^{T,n}|^{2}+|\tilde{R}_{z}^{n,\sigma^{n}}|^{2}%
+|R_{z}^{T,n,\sigma^{n}}|^{2}\right)  ,
\end{align*}
where we have chosen $\gamma=C^{2}/\underline{\sigma}^{2}$. It yields that
\begin{align}
\left \vert \delta Y_{n}\right \vert ^{2}+C\Delta t\left \vert \delta
Z_{n}\right \vert ^{2}  &  \leq(1+C\Delta t)\mathbb{E}_{t_{n}}^{TR,\sigma
^{n},X^{n}}\left[  |\delta Y_{n+1}|^{2}+C\Delta t\left \vert \delta
Z_{n+1}\right \vert ^{2}\right] \label{4.7}\\
&  \text{ \  \  \ }+\frac{C}{\Delta t}\left(  |\tilde{R}_{y}^{n}|^{2}%
+|R_{y}^{T,n}|^{2}+|\tilde{R}_{z}^{n,\sigma^{n}}|^{2}+|R_{z}^{T,n,\sigma^{n}%
}|^{2}\right)  .\quad \nonumber
\end{align}
Step 2. We consider the case of $Y_{t_{n}}^{t_{n},X^{n}}\geq Y^{n}$.
Subtracting (\ref{Yn_Tri}) from (\ref{4.1}) and noting that%
\begin{align}
&  \mathbb{E}_{t_{n}}^{TR,\tilde{\sigma}^{n},X^{n}}\left[  Y_{n+1}%
+f_{n+1}\Delta t+g_{n+1}\Delta \langle B\rangle_{n+1}\right] \\
&  \leq \mathbb{E}_{t_{n}}^{TR,\sigma^{n},X^{n}}\left[  Y_{n+1}+f_{n+1}\Delta
t+g_{n+1}\Delta \langle B\rangle_{n+1}\right]  ,\nonumber
\end{align}
it follows that
\begin{equation}
\delta Y_{n}\leq \mathbb{E}_{t_{n}}^{TR,\tilde{\sigma}^{n},X^{n}}\left[  \delta
Y_{n+1}+\delta f_{n+1}\Delta t+\delta g_{n+1}\Delta \langle B\rangle
_{n+1}\right]  +\tilde{R}_{y}^{n}+R_{y}^{T,n}.\text{ \  \  \ } \label{4.8}%
\end{equation}
In addition, by (\ref{Zn_Tri}) and (\ref{4.10}), we can see
\begin{align}
(\tilde{\sigma}^{n})^{2}\delta Z_{n}\Delta t  &  =(\tilde{\sigma}^{n}%
)^{2}\Delta t(Z_{t_{n}}^{t_{n},X^{n}}-\tilde{Z}^{n}+\tilde{Z}^{n}%
-Z^{n})\label{4.11}\\
&  =\mathbb{E}_{t_{n}}^{TR,\tilde{\sigma}^{n},X^{n}}\left[  \delta
Y_{n+1}\Delta B_{n+1}\right]  +(\tilde{\sigma}^{n})^{2}\Delta t(\tilde{Z}%
^{n}-Z^{n})+\tilde{R}_{z}^{n,\tilde{\sigma}^{n}}+R_{z}^{T,n,\tilde{\sigma}%
^{n}},\nonumber
\end{align}
where $\tilde{Z}^{n}$ is given in (\ref{Z^wan}). Based on (\ref{4.8}%
)-(\ref{4.11}), similar to the step 1, we can deduce
\begin{align}
\left \vert \delta Y_{n}\right \vert ^{2}+C\Delta t\left \vert \delta
Z_{n}\right \vert ^{2}  &  \leq(1+C\Delta t)\mathbb{E}_{t_{n}}^{TR,\tilde
{\sigma}^{n},X^{n}}\left[  |\delta Y_{n+1}|^{2}+C\Delta t\left \vert \delta
Z_{n+1}\right \vert ^{2}\right] \label{4.12}\\
&  \text{ \  \  \ }+\frac{C}{\Delta t}\left(  |\tilde{R}_{y}^{n}|^{2}%
+|R_{y}^{T,n}|^{2}+|\tilde{R}_{z}^{n,\tilde{\sigma}^{n}}|^{2}+|R_{z}%
^{T,n,\tilde{\sigma}^{n}}|^{2}\right) \nonumber \\
&  \text{ \  \  \ }+C\Delta t|\tilde{Z}^{n}-Z^{n}|^{2}.\nonumber
\end{align}
Step 3. Together with (\ref{4.7}) and (\ref{4.12}), we have
\begin{align}
\left \vert \delta Y_{n}\right \vert ^{2}+C\Delta t\left \vert \delta
Z_{n}\right \vert ^{2}  &  \leq(1+C\Delta t)\mathbb{\tilde{E}}_{t_{n}%
}^{TR,X^{n}}\left[  |\delta Y_{n+1}|^{2}+C\Delta t\left \vert \delta
Z_{n+1}\right \vert ^{2}\right]  \quad \label{4.13}\\
&  \text{ \  \  \ }+\frac{C}{\Delta t}\left(  |\tilde{R}_{y}^{n}|^{2}%
+|R_{y}^{T,n}|^{2}+|\tilde{R}_{z}^{n}|^{2}+|R_{z}^{T,n}|^{2}\right)
\quad \quad \nonumber \\
&  \text{ \  \  \ }+C\Delta t|\tilde{Z}^{n}-Z^{n}|^{2},\nonumber
\end{align}
where $\tilde{R}_{z}^{n}=\tilde{R}_{z}^{n,\sigma^{n}}\vee \tilde{R}%
_{z}^{n,\tilde{\sigma}^{n}}$\ and $R_{z}^{T,n}=R_{z}^{T,n,\sigma^{n}}\vee
R_{z}^{T,n,\tilde{\sigma}^{n}}$. By Proposition \ref{PropE} $\left(  a\right)
$, we have%
\begin{align*}
\mathbb{\tilde{E}}^{TR}\left[  \left \vert \delta Y_{n}\right \vert ^{2}+C\Delta
t\left \vert \delta Z_{n}\right \vert ^{2}\right]   &  \leq(1+C\Delta
t)\mathbb{\tilde{E}}^{TR}\left[  |\delta Y_{n+1}|^{2}+C\Delta t\left \vert
\delta Z_{n+1}\right \vert ^{2}\right] \\
&  \text{ \  \  \ }+\frac{C}{\Delta t}\mathbb{\tilde{E}}^{TR}\left[  |\tilde
{R}_{y}^{n}|^{2}+|R_{y}^{T,n}|^{2}+|\tilde{R}_{z}^{n}|^{2}+|R_{z}^{T,n}%
|^{2}\right] \\
&  \text{ \  \  \ }+C\Delta t\mathbb{\tilde{E}}^{TR}\left[  |\tilde{Z}^{n}%
-Z^{n}|^{2}\right]  .
\end{align*}
Using the induction method, one obtains
\begin{align*}
\mathbb{\tilde{E}}^{TR}\left[  \left \vert \delta Y_{n}\right \vert ^{2}+C\Delta
t\left \vert \delta Z_{n}\right \vert ^{2}\right]   &  \leq(1+C\Delta
t)^{N-n}\mathbb{\tilde{E}}^{TR}\left[  |\delta Y_{N}|^{2}+C\Delta t\left \vert
\delta Z_{N}\right \vert ^{2}\right] \\
&  \text{ \  \  \ }+\frac{C}{\Delta t}\sum \limits_{i=n}^{N-1}(1+C\Delta
t)^{i-n}\mathbb{\tilde{E}}^{TR}\left[  |\tilde{R}_{y}^{i}|^{2}+|R_{y}%
^{T,i}|^{2}+|\tilde{R}_{z}^{i}|^{2}+|R_{z}^{T,i}|^{2}\right] \\
&  \text{ \  \  \ }+C\Delta t\sum \limits_{i=n}^{N-1}(1+C\Delta t)^{i-n}%
\mathbb{\tilde{E}}^{TR}\left[  |\tilde{Z}^{i}-Z^{i}|^{2}\right]  ,
\end{align*}
which yields the result.
\end{proof}

\subsection{Error estimates}

To provide the error estimates for Scheme \ref{trinomial tree scheme}, we introduce the
following notations. For given constants $\alpha,\beta \in(0,1]$ and $Q=[0,T]
\times \mathbb{R}^{m}$, denote
\[
\left \Vert u\right \Vert _{C^{\alpha,\beta}(Q)}=\sup
_{\substack{x,y\in \mathbb{R}^{m},x\neq y\\s,t\in [0,T]  ,s\neq
t}}\frac{\left \vert u(s,x)-u(t,y) \right \vert
}{\left \vert s-t\right \vert ^{a}+\left \vert x-y\right \vert ^{\beta}},
\]
and%
\[%
\begin{array}
[c]{r}%
C_{b}^{1+\alpha,2+\beta}(Q)  =\left \{  u:\left \Vert \partial
_{t}u\right \Vert _{C^{\alpha,\beta}(Q)}+
{\textstyle \sum_{i=1}^{m}}
\left \Vert \partial_{x_{i}}u\right \Vert _{C^{\alpha,\beta}(Q)}+
{\textstyle \sum_{i,j=1}^{m}}
\left \Vert \partial_{x_{i}x_{j}}^{2}u\right \Vert _{C^{\alpha,\beta}(Q)}<\infty \right \}  .
\end{array}
\]
Similarly, we can define $C_{b}^{1+\alpha,2+\beta,2+\beta,2+\beta}(Q)$. We need the following assumptions:

\begin{enumerate}
\item[(A1)] The functions $f(t,x,y,z), g(t,x,y,z), b(t,x), \sigma(t,x)$, and
$h(t,x)$ are uniformly Lipschitz continuous with respect to $(x,y,z)$ and
H\"{o}lder continuous of parameter $\frac{1}{2}$ with respect to $t$;
$b,\sigma$, and $h$ are bounded;

\item[(A2)] The function $u(t,x) \in C_{b}^{1+1/2,2+1}([0,T] \times
\mathbb{R}^{m})$.
\end{enumerate}

\begin{remark}
If $f,g\in C_{b}^{1+1/2,2+1,2+1,2+1}([0,T] \times \mathbb{R}^{m}\times
\mathbb{R}\times \mathbb{R}) $ and $\phi \in C_{b}^{2+1}( \mathbb{R}^{m})$, by
Theorem 6.4.3 in Krylov \cite{Krylov1987} (see also Theorem 4.4 in Appendix C
in Peng \cite{P2010}), then there exists a constant $\alpha \in(0,1)$ such that
for each $\kappa>0$, $u\in C_{b}^{1+\alpha/2,2+\alpha}([0,T-\kappa
]\times \mathbb{R}^{m})$.
\end{remark}

We state our convergence theorem here.

\begin{theorem}
\label{theorem 2}Suppose \emph{(A1)}-\emph{(A2)} hold. Let $(X_{t}
^{t_{n},X^{n}},Y_{t}^{t_{n},X^{n}},Z_{t}^{t_{n},X^{n}})_{t_{n}\leq t\leq T}$
be the solution of \eqref{G-SDE}-\eqref{G-BSDE}, and $(X^{n},Y^{n}
,Z^{n})(n=0,1,\ldots,N)$ be the solution to Scheme \ref{trinomial tree scheme}
with $Y^{N}=u(t_{N},X^{N})$ and $Z^{N}=D_{x}u(t_{N},X^{N}) \sigma(t_{N}%
,X^{N})$. Then, for sufficiently small time step $\Delta t$,
\[
\mathbb{\tilde{E}}^{TR}\left[  |Y_{t_{n}}^{t_{n},X^{n}}-Y^{n}|^{2}+\Delta
t|Z_{t_{n}}^{t_{n},X^{n}}-Z^{n}|^{2}\right]  \leq C\Delta t.
\]

\end{theorem}

To prove the above theorem, we need the following Lemmas \ref{lemma1}%
-\ref{lemma5}.

\begin{lemma}
\label{lemma1}The numerical scheme \eqref{zn_tr} for $Z^{n}\ $is independent
of the parameter\ $\sigma$.
\end{lemma}

\begin{proof}
Notice that
\[
\frac{\mathbb{E}_{t_{n}}^{TR,\sigma,X^{n}}\left[  Y^{n+1}\Delta B_{n+1}%
\right]  }{\sigma^{2}\Delta t}=\frac{1}{2\sqrt{\Delta t}}\left[
Y^{n+1}\left(  X^{n+1}(q_{1})\right)  q_{1}+Y^{n+1}\left(  X^{n+1}%
(q_{3})\right)  q_{3}\right]  ,
\]
where
\[
X^{n+1}\left(  q_{i}\right)  =X^{n}+b(t_{n},X^{n})\Delta t+h(t_{n}%
,X^{n})\Delta tq_{i}^{2}+\sigma(t_{n},X^{n})\sqrt{\Delta t}q_{i},
\]
and $\left \{  q_{i}\right \}  _{i=1}^{3}$ is defined in (\ref{qi}),\ which
implies the result.
\end{proof}

\begin{lemma}
\label{lemma2}Suppose \emph{(A1)}-\emph{(A2)} hold. Let $\hat{R}_{y_{1}}^{n}$
and $\hat{R}_{y_{2}}^{n}$ be the truncation errors defined in \eqref{2.5} and
\eqref{2.6}, respectively. Then, for $n=0,1,\ldots,N-1,$
\[%
\begin{array}
[c]{ll}%
|\hat{R}_{y_{1}}^{n}|\leq C(\Delta t)^{\frac{3}{2}},\text{ \ } & |\hat
{R}_{y_{2}}^{n}|\leq C(\Delta t)^{\frac{3}{2}}.
\end{array}
\]

\end{lemma}

\begin{proof}
1. Recalling the property of $\langle B\rangle_{.}$ (see
\cite{P2010}, Corollary 3.5.5)
\begin{equation}
\underline{\sigma}^{2}(t-s)\leq \langle B\rangle_{t}-\langle B\rangle_{s}%
\leq \overline{\sigma}^{2}(t-s),\text{ \ for }0\leq s\leq t<\infty, \label{5.2}%
\end{equation}
one has%
\begin{align}
|\hat{R}_{y_{1}}^{n}|  &  \leq \mathbb{\hat{E}}_{t_{n}}^{X^{n}}\left[
\left \vert \int_{t_{n}}^{t_{n+1}}(f_{t_{n+1}}^{t_{n},X^{n}}-f_{t}^{t_{n}%
,X^{n}})dt+\int_{t_{n}}^{t_{n+1}}(g_{t_{n+1}}^{t_{n},X^{n}}-g_{t}^{t_{n}%
,X^{n}})d\langle B\rangle_{t}\right \vert \right] \label{5.1}\\
&  \leq \mathbb{\hat{E}}_{t_{n}}^{X^{n}}\left[  \int_{t_{n}}^{t_{n+1}%
}|f_{t_{n+1}}^{t_{n},X^{n}}-f_{t}^{t_{n},X^{n}}|dt\right]  +C\mathbb{\hat{E}%
}_{t_{n}}^{X^{n}}\left[  \int_{t_{n}}^{t_{n+1}}|g_{t_{n+1}}^{t_{n},X^{n}%
}-g_{t}^{t_{n},X^{n}}|dt\right]  .\nonumber
\end{align}
Under (A1), it is easy to check that
\begin{equation}
\mathbb{\hat{E}}_{t_{n}}^{X^{n}}\left[  |X_{t}^{t_{n},X^{n}}-X_{s}%
^{t_{n},X^{n}}|^{2}\right]  \leq C\Delta t,\text{ \ for }t_{n}\leq s\leq t\leq
t_{n+1}. \label{5.0}%
\end{equation}
Then, from the Lipschitz continuity of $f,u,D_{x}u$, and $\sigma$, we can
derive that
\begin{equation}
\mathbb{\hat{E}}_{t_{n}}^{X^{n}}\left[  |f_{t_{n+1}}^{t_{n},X^{n}}%
-f_{t}^{t_{n},X^{n}}|^{2}\right]  \leq C\left(  \left \vert t_{n+1}%
-t\right \vert +\mathbb{\hat{E}}_{t_{n}}^{X^{n}}\left[  |X_{t_{n+1}}%
^{t_{n},X^{n}}-X_{t}^{t_{n},X^{n}}|^{2}\right]  \right)  \leq C\Delta t,
\label{5.3.1}%
\end{equation}
and similarly,
\begin{equation}
\mathbb{\hat{E}}_{t_{n}}^{X^{n}}\left[  |g_{t_{n+1}}^{t_{n},X^{n}}%
-g_{t}^{t_{n},X^{n}}|^{2}\right]  \leq C\Delta t. \label{5.3.2}%
\end{equation}
Substituting (\ref{5.3.1})-(\ref{5.3.2}) into (\ref{5.1}), by the
Cauchy-Schwarz inequality, it is easy to check that $|\hat{R}_{y_{1}}^{n}%
|\leq$ $C(\Delta t)^{\frac{3}{2}}$.

2. From the sublinear property of $\mathbb{\hat{E}}_{t_{n}}^{X^{n}}[\cdot]$, we have
\begin{equation}
|\hat{R}_{y_{2}}^{n}|\leq \mathbb{\hat{E}}_{t_{n}}^{X^{n}}\left[  M_{t_{n+1}%
}\right]  \vee \mathbb{\hat{E}}_{t_{n}}^{X^{n}}\left[  -M_{t_{n+1}}\right]  ,
\label{5.4}%
\end{equation}
where%
\[
M_{t_{n+1}}=Y_{t_{n+1}}^{t_{n},X^{n}}-Y_{t_{n+1}}^{t_{n+1},X^{n+1}%
}+(f_{t_{n+1}}^{t_{n},X^{n}}-f_{t_{n+1}}^{t_{n+1},X^{n+1}})\Delta
t+(g_{t_{n+1}}^{t_{n},X^{n}}-g_{t_{n+1}}^{t_{n+1},X^{n+1}})\Delta \langle
B\rangle_{n+1}.
\]
Next, we only bound the first term on the right-hand side of (\ref{5.4}), and
the second term can be similarly obtained.\ Define the continuous-time
approximation
\[
\tilde{X}_{t}^{t_{n},X^{n}}=X^{n}+\int_{t_{n}}^{t}b(t_{n},X^{n})dt+\int
_{t_{n}}^{t}h(t_{n},X^{n})d\langle B\rangle_{t}+\int_{t_{n}}^{t}\sigma
(t_{n},X^{n})dB_{t},
\]
for $t\in \left[  t_{n},t_{n+1}\right]  $. Similar to (\ref{5.0}), we have for
$t_{n}\leq s\leq t\leq t_{n+1},$
\begin{equation}
\mathbb{\hat{E}}_{t_{n}}^{X^{n}}\left[  |\tilde{X}_{t}^{t_{n},X^{n}}-\tilde
{X}_{s}^{t_{n},X^{n}}|^{2}\right]  \leq C\Delta t. \label{5.6}%
\end{equation}
Seeing that $Y_{t_{n+1}}^{t_{n},X^{n}}=u(t_{n+1},X_{t_{n+1}}^{t_{n},X^{n}})$
and $Y_{t_{n+1}}^{t_{n+1},X^{n+1}}=u(t_{n+1},\tilde{X}_{t_{n+1}}^{t_{n},X^{n}%
})$, by $G$-It\^{o}'s formula and (\ref{5.2}), we get
\begin{align}
&  \mathbb{\hat{E}}_{t_{n}}^{X^{n}}\left[  Y_{t_{n+1}}^{t_{n},X^{n}%
}-Y_{t_{n+1}}^{t_{n+1},X^{n+1}}\right] \label{5.5}\\
&  \leq \mathbb{\hat{E}}_{t_{n}}^{X^{n}}\left[  \int_{t_{n}}^{t_{n+1}}%
\{L^{0}u(t,X_{t}^{t_{n},X^{n}})-\tilde{L}^{0}u(t,\tilde{X}_{t}^{t_{n},X^{n}%
})\}dt\right]  \text{ }\nonumber \\
&  \text{ \  \ }+\mathbb{\hat{E}}_{t_{n}}^{X^{n}}\left[  \int_{t_{n}}^{t_{n+1}%
}\{L^{1}u(t,X_{t}^{t_{n},X^{n}})-\tilde{L}^{1}u(t,\tilde{X}_{t}^{t_{n},X^{n}%
})\}d\langle B\rangle_{t}\right]  ,\text{ \  \  \  \  \  \ }\nonumber
\end{align}
where%
\begin{align*}
L^{0}  &  =\partial_{t}+\sum_{i=1}^{m}b_{i}\left(  t,X_{t}\right)
\partial_{x_{i}}\text{, \ }L^{1}=\sum_{i=1}^{m}h_{i}\left(  t,X_{t}\right)
\partial_{x_{i}}+\frac{1}{2}\sum_{i,j=1}^{m}[\sigma \sigma^{\top}]_{i,j}\left(
t,X_{t}\right)  \partial_{x_{i}x_{j}}^{2}\text{,}\\
\tilde{L}^{0}  &  =\partial_{t}+\sum_{i=1}^{m}b_{i}\left(  t_{n},X^{n}\right)
\partial_{x_{i}}\text{, \ }\tilde{L}^{1}=\sum_{i=1}^{m}h_{i}\left(
t_{n},X^{n}\right)  \partial_{x_{i}}+\frac{1}{2}\sum_{i,j=1}^{m}[\sigma
\sigma^{\top}]_{i,j}\left(  t_{n},X^{n}\right)  \partial_{x_{i}x_{j}}^{2}\text{.}
\end{align*}
From the definition of $L^{0}$ and $\tilde{L}^{0}$, we can obtain%
\begin{align}
&  \mathbb{\hat{E}}_{t_{n}}^{X^{n}}\left[  \int_{t_{n}}^{t_{n+1}}%
\{L^{0}u(t,X_{t}^{t_{n},X^{n}})-\tilde{L}^{0}u(t,\tilde{X}_{t}^{t_{n},X^{n}%
})\}dt\right] \label{5.7}\\
&  \leq \mathbb{\hat{E}}_{t_{n}}^{X^{n}}\left[  \int_{t_{n}}^{t_{n+1}}\{
\partial_{t}u(t,X_{t}^{t_{n},X^{n}})-\partial_{t}u(t,\tilde{X}_{t}%
^{t_{n},X^{n}})\}dt\right] \nonumber \\
&  \text{ \  \ }+\mathbb{\hat{E}}_{t_{n}}^{X^{n}}\left[  \int_{t_{n}}^{t_{n+1}%
}D_{x}u(t,X_{t}^{t_{n},X^{n}})\{b(t,X_{t}^{t_{n},X^{n}})-b(t_{n}%
,X^{n})\}dt\right] \nonumber \\
&  \text{ \  \ }+\mathbb{\hat{E}}_{t_{n}}^{X^{n}}\left[  \int_{t_{n}}^{t_{n+1}%
}\{D_{x}u(t,X_{t}^{t_{n},X^{n}})-D_{x}u(t,\tilde{X}_{t}^{t_{n},X^{n}%
})\}b(t_{n},X^{n})dt\right]  .\nonumber
\end{align}
Under (A1)-(A2), by (\ref{5.0}), (\ref{5.6}) and the
Cauchy-Schwarz inequality, we have
\begin{align}
&  \mathbb{\hat{E}}_{t_{n}}^{X^{n}}\left[  \int_{t_{n}}^{t_{n+1}}%
\{L^{0}u(t,X_{t}^{t_{n},X^{n}})-\tilde{L}^{0}u(t,\tilde{X}_{t}^{t_{n},X^{n}%
})\}dt\right] \label{5.9.1}\\
&  \leq C\int_{t_{n}}^{t_{n+1}}\left \{  \mathbb{\hat{E}}_{t_{n}}^{X^{n}%
}\left[  |X_{t}^{t_{n},X^{n}}-X^{n}|\right]  +\mathbb{\hat{E}}_{t_{n}}^{X^{n}%
}\left[  |\tilde{X}_{t}^{t_{n},X^{n}}-X^{n}|\right]  \right \}  dt\text{
\  \  \ }\nonumber \\
&  \leq C(\Delta t)^{\frac{3}{2}}.\nonumber
\end{align}
In the same way as above, we can obtain
\begin{equation}
\mathbb{\hat{E}}_{t_{n}}^{X^{n}}\left[  \int_{t_{n}}^{t_{n+1}}\{L^{1}%
u(t,X_{t}^{t_{n},X^{n}})-\tilde{L}^{1}u(t,\tilde{X}_{t}^{t_{n},X^{n}%
})\}d\langle B\rangle_{t}\right]  \leq C(\Delta t)^{\frac{3}{2}}.\text{ }
\label{5.9.2}%
\end{equation}
Together with (\ref{5.5}), (\ref{5.9.1}), and (\ref{5.9.2}), we have
\[
\mathbb{\hat{E}}_{t_{n}}^{X^{n}}\left[  Y_{t_{n+1}}^{t_{n},X^{n}}-Y_{t_{n+1}%
}^{t_{n+1},X^{n+1}}\right]  \leq C(\Delta t)^{\frac{3}{2}}.
\]
Using the Lipschitz condition of $f,g$, and $u$, we can also estimate%
\[%
\begin{array}
[c]{rr}%
\mathbb{\hat{E}}_{t_{n}}^{X^{n}}\left[  (f_{t_{n+1}}^{t_{n},X^{n}}-f_{t_{n+1}
}^{t_{n+1},X^{n+1}})\Delta t\right]  \leq C(\Delta t)^{\frac{3}{2}}, &
\mathbb{\hat{E}}_{t_{n}}^{X^{n}}\left[  (g_{t_{n+1}}^{t_{n},X^{n}}-g_{t_{n+1}
}^{t_{n+1},X^{n+1}})\Delta \langle B\rangle_{n+1}\right]  \leq C(\Delta
t)^{\frac{3}{2}}.
\end{array}
\]
Thus our conclusion follows.
\end{proof}

\begin{lemma}
\label{lemma3}Suppose \emph{(A1)}-\emph{(A2)} hold. Let $R_{z_{1}}^{n,\sigma
},R_{z_{2}}^{n,\sigma}$, and $R_{z_{3}}^{n,\sigma}$, $\sigma \in \left \{
\underline{\sigma},\bar{\sigma}\right \}  $, be the truncation errors defined
in \eqref{2.10}, \eqref{2.11}, and \eqref{2.14}, respectively. Then, for
$n=0,1,\ldots,N-1,$%
\[%
\begin{array}
[c]{lll}%
|R_{z_{1}}^{n,\sigma}|\leq C(\Delta t)^{\frac{3}{2}},\text{\ } & |R_{z_{2}%
}^{n,\sigma}|\leq C(\Delta t)^{\frac{3}{2}},\text{ } & |R_{z3}^{n,\sigma}|\leq
C(\Delta t)^{\frac{3}{2}}.
\end{array}
\]

\end{lemma}

\begin{proof}
Under (A1)-(A2) and Remark \ref{F-K represent}, the estimates of
$R_{z_{1}}^{n,\sigma},R_{z_{2}}^{n,\sigma}$, and $R_{z3}^{n,\sigma}$ can be
obtained similarly to Lemma \ref{lemma2}.
\end{proof}

For given $(t_{n},X^{n})$, we define the following functions
\begin{equation}%
\begin{array}
[c]{l}%
\xi_{n}\left(  x,y;X^{n}\right)  =X^{n}+b\left(  t_{n},X^{n}\right)  \Delta
t+\sigma \left(  t_{n},X^{n}\right)  x+h\left(  t_{n},X^{n}\right)  y,\\
u_{n}\left(  x,y;X^{n}\right)  =u\left(  t_{n+1},\xi_{n}\left(  x,y;X^{n}%
\right)  \right)  ,\\
v_{n}\left(  x,y;X^{n}\right)  =D_{x}u\left(  t_{n+1},\xi_{n}\left(
x,y;X^{n}\right)  \right)  \sigma \left(  t_{n+1},\xi_{n}\left(  x,y;X^{n}%
\right)  \right)  ,\\
f_{n}\left(  x,y;X^{n}\right)  =f\left(  t_{n+1},\xi_{n}\left(  x,y;X^{n}%
\right)  ,u_{n}\left(  x,y;X^{n}\right)  ,v_{n}\left(  x,y;X^{n}\right)
\right)  ,\\
g_{n}\left(  x,y;X^{n}\right)  =g\left(  t_{n+1},\xi_{n}\left(  x,y;X^{n}%
\right)  ,u_{n}\left(  x,y;X^{n}\right)  ,v_{n}\left(  x,y;X^{n}\right)
\right)  ,\\
\varphi_{n}\left(  x,y;X^{n}\right)  =u_{n}\left(  x,y;X^{n}\right)
+f_{n}\left(  x,y;X^{n}\right)  \Delta t+g_{n}\left(  x,y;X^{n}\right)  y.
\end{array}
\label{7.0}%
\end{equation}
For convenience, we will omit $X^{n}$ in the following proof, if no ambiguity arises.

\begin{lemma}
\label{lemma4}Suppose \emph{(A1)}-\emph{(A2)} hold. Let $\hat{R}_{y_{3}}^{n}$
be the error defined in \eqref{2.8}. Then, for $n=0,1,\ldots,N-1,$
\[
|\hat{R}_{y_{3}}^{n}|\leq C(\Delta t)^{\frac{3}{2}}.
\]

\end{lemma}

\begin{proof}
Noting that
\[
\varphi_{n}\left(  \Delta B_{n+1},\Delta \langle B\rangle_{n+1}\right)
\overset{d}{=}\varphi_{n}\left(  B_{\Delta t},\langle B\rangle_{\Delta
t}\right)  ,
\]
in view of (\ref{7.0}), we can rewrite $\hat{R}_{y_{3}}^{n}$ as%
\begin{equation}
\hat{R}_{y_{3}}^{n}=\mathbb{\hat{E}}_{t_{n}}^{X^{n}}\left[  \varphi_{n}\left(
B_{\Delta t},\langle B\rangle_{\Delta t}\right)  \right]  -\mathbb{\tilde{E}%
}_{t_{n}}^{X^{n}}\left[  \varphi_{n}\left(  B_{\Delta t},\langle
B\rangle_{\Delta t}\right)  \right]  . \label{5.8}%
\end{equation}
Applying $G$-It\^{o}'s formula to $u_{n}\left(  B_{\Delta t},\langle
B\rangle_{\Delta t}\right)  $, we have%
\begin{align*}
u_{n}\left(  B_{\Delta t},\langle B\rangle_{\Delta t}\right)   &
=u_{n}\left(  0,0\right)  +\int_{0}^{\Delta t}\partial_{x}u_{n}\left(
B_{s},\langle B\rangle_{s}\right)  dB_{s}\\
&  \text{ \  \ }+\int_{0}^{\Delta t}[\partial_{y}u_{n}\left(  B_{s},\langle
B\rangle_{s}\right)  +\frac{1}{2}\partial_{xx}^{2}u_{n}\left(  B_{s},\langle
B\rangle_{s}\right)  ]d\langle B\rangle_{s}.\text{ \ }%
\end{align*}
Noting that $\mathbb{\hat{E}}\left[  \langle B\rangle_{\Delta t}\right]
=\bar{\sigma}^{2}\Delta t$ and $-\mathbb{\hat{E}}\left[  -\langle
B\rangle_{\Delta t}\right]  =\underline{\sigma}^{2}\Delta t$, we obtain
\begin{align}
&  \mathbb{\hat{E}}_{t_{n}}^{X^{n}}\left[  \varphi_{n}\left(  B_{\Delta
t},\langle B\rangle_{\Delta t}\right)  \right] \label{5.9}\\
&  =\varphi_{n}\left(  0,0\right)  +\mathbb{\hat{E}}_{t_{n}}^{X^{n}}\left[
\{g_{n}\left(  0,0\right)  +\partial_{y}u_{n}\left(  0,0\right)  +\frac{1}%
{2}\partial_{xx}^{2}u_{n}\left(  0,0\right)  \} \langle B\rangle_{\Delta
t}\right]  +\mathcal{E}_{n}\nonumber \\
&  =\varphi_{n}\left(  0,0\right)  +G\left(  2g_{n}\left(  0,0\right)
+2\partial_{y}u_{n}\left(  0,0\right)  +\partial_{xx}^{2}u_{n}\left(
0,0\right)  \right)  \Delta t+\mathcal{E}_{n},\nonumber
\end{align}
where%
\begin{align}
\mathcal{E}_{n}  &  =\mathbb{\hat{E}}_{t_{n}}^{X^{n}}\bigg[\int_{0}^{\Delta
t}\{ \partial_{y}u_{n}\left(  B_{s},\langle B\rangle_{s}\right)  +\frac{1}%
{2}\partial_{xx}^{2}u_{n}\left(  B_{s},\langle B\rangle_{s}\right)  \}d\langle
B\rangle_{s}\label{5.10}\\
&  \text{ \  \ }+g_{n}\left(  B_{\Delta t},\langle B\rangle_{\Delta t}\right)
\langle B\rangle_{\Delta t}+f_{n}\left(  B_{\Delta t},\langle B\rangle_{\Delta
t}\right)  \Delta t\bigg]-f_{n}\left(  0,0\right)  \Delta t\text{
\  \  \  \  \  \ }\nonumber \\
&  \text{ \  \ }-\mathbb{\hat{E}}_{t_{n}}^{X^{n}}\left[  \{g_{n}\left(
0,0\right)  +\partial_{y}u_{n}\left(  0,0\right)  +\frac{1}{2}\partial
_{xx}^{2}u_{n}\left(  0,0\right)  \} \langle B\rangle_{\Delta t}\right]
.\nonumber
\end{align}
With the help of (\ref{5.2}), one can check that%
\begin{align*}
\left \vert \mathcal{E}_{n}\right \vert  &  \leq C\mathbb{\hat{E}}_{t_{n}%
}^{X^{n}}\left[  \int_{0}^{\Delta t}\left \vert f_{n}(B_{\Delta t},\langle
B\rangle_{\Delta t})-f_{n}\left(  0,0\right)  \right \vert +|g_{n}(B_{\Delta
t},\langle B\rangle_{\Delta t})-g_{n}\left(  0,0\right)  |ds\right] \\
&  \text{ \  \ }+C\mathbb{\hat{E}}_{t_{n}}^{X^{n}}\left[  \int_{0}^{\Delta
t}\left \vert \partial_{y}u_{n}(B_{s},\langle B\rangle_{s})-\partial_{y}%
u_{n}\left(  0,0\right)  \right \vert ds\right] \\
&  \text{ \  \ }+C\mathbb{\hat{E}}_{t_{n}}^{X^{n}}\left[  \int_{0}^{\Delta
t}\left \vert \partial_{xx}^{2}u_{n}(B_{s},\langle B\rangle_{s})-\partial
_{xx}^{2}u_{n}\left(  0,0\right)  \right \vert ds\right]  .
\end{align*}
In view of the assumptions (A1)-(A2), we get $\vert \mathcal{E}_{n}\vert \leq
C(\Delta t)^{3/2}$.\ Thus
\begin{align}
&  \mathbb{\hat{E}}_{t_{n}}^{X^{n}}\left[  \varphi_{n}\left(  B_{\Delta
t},\langle B\rangle_{\Delta t}\right)  \right] \label{5.12}\\
&  =\varphi_{n}\left(  0,0\right)  +G\left(  2g_{n}\left(  0,0\right)
+2\partial_{y}u_{n}\left(  0,0\right)  +\partial_{xx}^{2}u_{n}\left(
0,0\right)  \right)  \Delta t+\mathcal{O}(\Delta t)^{\frac{3}{2}}.\quad
\quad \quad \quad \nonumber
\end{align}
Similarly, we have
\begin{align}
&  \mathbb{\tilde{E}}_{t_{n}}^{X^{n}}\left[  \varphi_{n}\left(  B_{\Delta
t},\langle B\rangle_{\Delta t}\right)  \right] \label{5.13}\\
&  =\varphi_{n}\left(  0,0\right)  +\sup_{\sigma \in \left \{  \underline{\sigma
},\bar{\sigma}\right \}  }\mathbb{E}_{t_{n}}^{\sigma,X^{n}}\left[
\{g_{n}\left(  0,0\right)  +\partial_{y}u_{n}\left(  0,0\right)  +\frac{1}%
{2}\partial_{xx}^{2}u_{n}\left(  0,0\right)  \} \langle B\rangle_{\Delta
t}\right]  +\mathcal{\tilde{E}}_{n}\nonumber \\
&  =\varphi_{n}\left(  0,0\right)  +G\left(  2g_{n}\left(  0,0\right)
+2\partial_{y}u_{n}\left(  0,0\right)  +\partial_{xx}^{2}u_{n}\left(
0,0\right)  \right)  \Delta t+\mathcal{\tilde{E}}_{n},\nonumber
\end{align}
where%
\begin{align*}
\mathcal{\tilde{E}}_{n}  &  =\mathbb{\tilde{E}}_{t_{n}}^{X^{n}}\bigg[\int
_{0}^{\Delta t}\{ \partial_{y}u_{n}\left(  B_{s},\langle B\rangle_{s}\right)
+\frac{1}{2}\partial_{xx}^{2}u_{n}\left(  B_{s},\langle B\rangle_{s}\right)
\}d\langle B\rangle_{s}\\
&  \text{ \  \ }+g_{n}\left(  B_{\Delta t},\langle B\rangle_{\Delta t}\right)
\langle B\rangle_{\Delta t}+f_{n}\left(  B_{\Delta t},\langle B\rangle_{\Delta
t}\right)  \Delta t\bigg]-f_{n}\left(  0,0\right)  \Delta t\\
&  \text{ \  \ }-\mathbb{\tilde{E}}_{t_{n}}^{X^{n}}\left[  \{g_{n}\left(
0,0\right)  +\partial_{y}u_{n}\left(  0,0\right)  +\frac{1}{2}\partial
_{xx}^{2}u_{n}\left(  0,0\right)  \} \langle B\rangle_{\Delta t}\right]  .
\end{align*}
By a simple calculation similar to $\mathcal{E}_{n}$, we have
$|\mathcal{\tilde{E}}_{n}|\leq C(\Delta t)^{3/2}$.
Together with (\ref{5.8}), (\ref{5.12}) and (\ref{5.13}), the desired result follows.
\end{proof}

\begin{lemma}
\label{lemma5}Suppose \emph{(A1)}-\emph{(A2)} hold. Let $R_{y}^{T,n}$ and
$R_{z}^{T,n,\sigma},\sigma \in \{ \underline{\sigma},\bar{\sigma}\}$ be the
errors defined in \eqref{3.7} and \eqref{3.8}, respectively. Then for
$n=0,1,\ldots,N-1,$
\[%
\begin{array}
[c]{rr}%
|R_{y}^{T,n}|\leq C(\Delta t)^{\frac{3}{2}},\text{ } & |R_{z}^{T,n,\sigma
}|\leq C(\Delta t)^{\frac{3}{2}}.
\end{array}
\]

\end{lemma}

\begin{proof}
From (\ref{7.0}), it is easy to know
\begin{equation}
R_{y}^{T,n}=\mathbb{\tilde{E}}_{t_{n}}^{X^{n}}\left[  \varphi_{n}\left(
B_{\Delta t},\langle B\rangle_{\Delta t}\right)  \right]  -\mathbb{\tilde{E}%
}_{t_{n}}^{TR,X^{n}}\left[  \varphi_{n}\left(  B_{\Delta t},\langle
B\rangle_{\Delta t}\right)  \right]  . \label{5.16}%
\end{equation}
Under the assumptions (A1)-(A2), applying the Taylor expansion to $u_{n}$ and
using the Lipschitz condition on $f_{n}$ and $g_{n}$, we get
\begin{align}
\varphi_{n}(\sqrt{\Delta t}q_{i},\Delta tq_{i}^{2})  &  =u_{n}\left(
0,0\right)  +\partial_{x}u_{n}\left(  0,0\right)  \sqrt{\Delta t}%
q_{i}+\partial_{y}u_{n}\left(  0,0\right)  \Delta tq_{i}^{2}\\
&  \text{ \  \ }+\frac{1}{2}\partial_{xx}^{2}u_{n}\left(  0,0\right)
(\sqrt{\Delta t}q_{i})^{2}+f_{n}\left(  0,0\right)  \Delta t\nonumber \\
&  \text{ \  \ }+g_{n}\left(  0,0\right)  \Delta tq_{i}^{2}+\mathcal{O}(\Delta
t)^{\frac{3}{2}}.\nonumber
\end{align}
This and Remark \ref{remark2} yield
\begin{align}
&  \mathbb{\tilde{E}}_{t_{n}}^{TR,X^{n}}\left[  \varphi_{n}\left(  B_{\Delta
t},\langle B\rangle_{\Delta t}\right)  \right] \label{5.14}\\
&  =\varphi_{n}\left(  0,0\right)  +\sup \limits_{\sigma \in \left \{
\underline{\sigma},\overline{\sigma}\right \}  }\left[  \{g_{n}\left(
0,0\right)  +\partial_{y}u_{n}\left(  0,0\right)  +\frac{1}{2}\partial
_{xx}^{2}u_{n}\left(  0,0\right)  \} \sigma^{2}\Delta t\right]  +\mathcal{O}%
(\Delta t)^{\frac{3}{2}}\nonumber \\
&  =\varphi_{n}\left(  0,0\right)  +G\left(  2g_{n}\left(  0,0\right)
+2\partial_{y}u_{n}\left(  0,0\right)  +\partial_{xx}^{2}u_{n}\left(
0,0\right)  \right)  \Delta t+\mathcal{O}(\Delta t)^{\frac{3}{2}}.\nonumber
\end{align}
Together with (\ref{5.13})-(\ref{5.16}) and (\ref{5.14}), we obtain
$|R_{y}^{T,n}|\leq C(\Delta t)^{\frac{3}{2}}$. Similarly, according to
(\ref{E^wan}) and (\ref{E TR}), the estimate for $R_{z}^{T,n,\sigma}$ follows.
\end{proof}

\begin{proof}
[Proof of Theorem \ref{theorem 2}]From Lemmas \ref{lemma2}-\ref{lemma5}, it
holds that%
\[%
\begin{array}
[c]{rr}%
|\tilde{R}_{y}^{n}|^{2}\leq C(\Delta t)^{3}, & |R_{y}^{T,n}|^{2}\leq C(\Delta
t)^{3};\\
|\tilde{R}_{z}^{n}|^{2}\leq C(\Delta t)^{3}, & |R_{z}^{T,n}|^{2}\leq C(\Delta
t)^{3}.
\end{array}
\]
Then by Theorem \ref{theorem 1} and Lemma \ref{lemma1}, the conclusion can be obtained.
\end{proof}

\begin{remark}
Suppose that $f,g\in C_{b}^{1+1/2,2+1,2+1,2+1}([0,T]\times \mathbb{R}^{m}%
\times \mathbb{R}\times \mathbb{R})$ and $\phi \in C_{b}^{2+1}(\mathbb{R}^{m})$.
Without \emph{(A2)}, we conclude that
\[
\mathbb{\tilde{E}}^{TR}\left[  |Y_{t_{n}}^{t_{n},X^{n}}-Y^{n}|^{2}\right]
\leq C(\Delta t)^{\alpha}, \;\;\text{for some }\alpha \in(0,1),
\]
which yields Scheme \ref{trinomial tree scheme} is of order $\alpha/2$ for $Y$.
\end{remark}

\section{Multi-dimensional $G$-Brownian motion case\label{multi case}}

In this section, we extend our results to the multi-dimensional $G$-Brownian
motion case. Let $G:$ $\mathbb{S}(d)\rightarrow \mathbb{R}$ be a given
sublinear function such that%
\begin{equation}
G\left(  A\right)  =\frac{1}{2}\sup_{Q\in \Sigma}tr[QA], \label{G}
\end{equation}
with the closed set $\Sigma \subset$ $\mathbb{S}^{+}(d)$, and $B_{t}=(B_{t}%
^{1},\ldots,B_{t}^{d})^{\top}$ be the corresponding $d$-dimensional
$G$-Brownian motion. Denote $h(t,x):=(h_{ij}(t,x))_{i,j=1}^{d}$ and
$g(t,x,y,z):=(g_{ij}(t,x,y,z))_{i,j=1}^{d}$. For $\eta \in L_{G}^{1}(\Omega
_{T})$, define the approximate conditional $G$-expectation $\mathbb{\tilde{E}%
}_{t_{n}}^{X^{n}}[\eta]$ by
\begin{equation}
\mathbb{\tilde{E}}_{t_{n}}^{X^{n}}[\eta]=\sup \limits_{Q\in \Sigma}%
\mathbb{E}_{t_{n}}^{Q,X^{n}}\left[  \eta \right]  =\sup_{Q\in \Sigma}%
\mathbb{E}_{P^{Q}}\left[  \left.  \eta \right \vert X_{t_{n}}=X^{n}\right]  ,
\end{equation}
where $P^{Q}\in \mathcal{P}$ is a probability measure, under which
$(B_{t})_{t\geq0}$ is the classical Brownian motion with $\mathbb{E}_{P^{Q}%
}[B_{t}]=0$ and $\mathbb{E}_{P^{Q}}[B_{t}B_{t}^{\top}]=Qt$. Then the equations
(\ref{ref1})-(\ref{ref2}) can be extended to
\begin{align}
&  Y_{t_{n}}^{t_{n},X^{n}}=\mathbb{\tilde{E}}_{t_{n}}^{X^{n}}\left[
Y_{t_{n+1}}^{t_{n+1},X^{n+1}}+f_{t_{n+1}}^{t_{n+1},X^{n+1}}\Delta t+\langle
g_{t_{n+1}}^{t_{n+1},X^{n+1}},\Delta \langle B\rangle_{n+1}\rangle \right]
+\tilde{R}_{y}^{n},\label{7.1}\\
&  Z_{t_{n}}^{t_{n},X^{n}}Q\Delta t=\mathbb{E}_{t_{n}}^{Q,X^{n}}\left[
Y_{t_{n+1}}^{t_{n+1},X^{n+1}}\Delta B_{n+1}^{\top}\right]  +\tilde{R}%
_{z}^{n,Q},\text{ for}\ Q\in \Sigma, \label{7.2}%
\end{align}
where $Z_{t}=(Z_{t}^{1},\ldots,Z_{t}^{d})$, $\tilde{R}_{y}^{n}$ is given in
(\ref{2.7}) and $\tilde{R}_{z}^{n,Q}=R_{z_{1}}^{n,Q}+R_{z_{2}}^{n,Q}-R_{z_{3}%
}^{n,Q}$ with%
\begin{align}
R_{z_{1}}^{n,Q}  &  =\mathbb{E}_{t_{n}}^{Q,X^{n}}\left[
{\textstyle \int_{t_{n}}^{t_{n+1}}}
f_{t}^{t_{n},X^{n}}dt\Delta B_{n+1}^{\top}+%
{\textstyle \int \nolimits_{t_{n}}^{t_{n+1}}}
g_{t}^{t_{n},X^{n}}d\langle B\rangle_{t}\Delta B_{n+1}^{\top}\right] \\
&  \text{ \  \ }-\mathbb{E}_{t_{n}}^{Q,X^{n}}\left[  (K_{t_{n+1}}^{t_{n},X^{n}%
}-K_{t_{n}}^{t_{n},X^{n}})\Delta B_{n+1}^{\top}\right]  ,\nonumber \\
R_{z_{2}}^{n,Q}  &  =\mathbb{E}_{t_{n}}^{Q,X^{n}}\left[  Y_{t_{n+1}}%
^{t_{n},X^{n}}\Delta B_{n+1}^{\top}\right]  -\mathbb{E}_{t_{n}}^{Q,X^{n}%
}\left[  Y_{t_{n+1}}^{t_{n+1},X^{n+1}}\Delta B_{n+1}^{\top}\right]  ,\\
R_{z_{3}}^{n,Q}  &  =\mathbb{E}_{t_{n}}^{Q,X^{n}}\left[
{\textstyle \int_{t_{n}}^{t_{n+1}}}
Z_{t}^{t_{n},X^{n}}Qdt\right]  -Z_{t_{n}}^{t_{n},X^{n}}Q\Delta t.
\end{align}

\begin{example}
\label{example1}Assume that $d=2$ and $Q=diag\left(  \sigma_{1}^{2},\sigma
_{2}^{2}\right)  $. In this case, we have
\[
\mathbb{E}_{t_{n}}^{Q,X^{n}}\left[  \varphi \left(  \Delta B_{n+1}^{1},\Delta
B_{n+1}^{2}\right)  \Delta B_{n+1}^{1}\right]  =\mathbb{E}_{t_{n}}^{\sigma
_{1},X^{n}}\left[  \mathbb{E}_{t_{n}}^{\sigma_{2},X^{n}}\left[  \varphi \left(
x,\Delta B_{n+1}^{2}\right)  x\right]  _{x=\Delta B_{n+1}^{1}}\right]
\text{.}%
\]
From the trinomial tree rule $\left(  \ref{E TR}\right)  $, we can deduce
that
\begin{align*}
\frac{\mathbb{E}_{t_{n}}^{TR,Q,X^{n}}\left[  \varphi \left(  \Delta
B_{n+1}\right)  \Delta B_{n+1}^{1}\right]  }{\sigma_{1}^{2}\Delta t}  &
=\frac{1}{\sigma_{1}^{2}\sqrt{\Delta t}}\sum_{i=1}^{3}\sum_{j=1}^{3}\omega
_{i}^{\sigma_{1}}\omega_{j}^{\sigma_{2}}\varphi(\sqrt{\Delta t}q_{i}%
,\sqrt{\Delta t}q_{j})q_{i}\\
&  =\frac{1}{2\sqrt{\Delta t}}\sum_{j=1}^{3}\omega_{j}^{\sigma_{2}}\left[
\varphi(\sqrt{\Delta t},\sqrt{\Delta t}q_{j})-\varphi(-\sqrt{\Delta t}%
,\sqrt{\Delta t}q_{j})\right]  ,
\end{align*}
which is independent of $\sigma_{1}$ but dependent on $\sigma_{2}$.
Similarly,
\[
\frac{\mathbb{E}_{t_{n}}^{TR,Q,X^{n}}\left[  \varphi \left(  \Delta
B_{n+1}\right)  \Delta B_{n+1}^{2}\right]  }{\sigma_{2}^{2}\Delta t}=\frac
{1}{2\sqrt{\Delta t}}\sum_{i=1}^{3}\omega_{i}^{\sigma_{1}}\left[
\varphi(\sqrt{\Delta t}q_{i},\sqrt{\Delta t})-\varphi(\sqrt{\Delta t}%
q_{i},-\sqrt{\Delta t})\right]  ,
\]
which is independent of $\sigma_{2}$ but dependent on $\sigma_{1}$.
\end{example}

\begin{remark}
Example \ref{example1} implies that the trinomial tree rule \eqref{zn_tr} for
$Z^{n}$ depends on $Q$ in the case of multi-dimensional $G$-Brownian motion,
and Lemma \ref{lemma1} is no longer true. In addition, in this case the
weights of the trinomial tree rule are determined by the matrix $Q\in \Sigma$,
which makes it more difficult to derive.
\end{remark}

\subsection{Gauss-Hermite quadrature rule}

Inspired by the approximation of classical $d$-dimensional Brownian motion
distribution, for a function $\varphi:\mathbb{R}^{m}\times \mathbb{R}^{d}%
\times \mathbb{R}^{d\times d}\rightarrow \mathbb{R}$, denote $\varphi_{t_{n+1}%
}=\varphi \left(  X^{n+1},\Delta B_{n+1},\Delta \langle B\rangle_{n+1}\right)
$, we introduce a more effective Gauss-Hermite quadrature rule to approximate
$\mathbb{\tilde{E}}_{t_{n}}^{GH,X^{n}}[\varphi_{t_{n+1}}]$ as follows:
\begin{align}
\mathbb{\tilde{E}}_{t_{n}}^{GH,X^{n}}[\varphi_{t_{n+1}}]  &  :=\sup
\limits_{Q\in \Sigma}\mathbb{E}_{t_{n}}^{GH,Q,X^{n}}\left[  \varphi \left(
X^{n+1},\Delta B_{n+1},\Delta \langle B\rangle_{n+1}\right)  \right] \\
&  :=\sup \limits_{Q\in \Sigma}\sum_{i_{1},\ldots,i_{d}=1}^{L}\omega_{i_{1}%
}\cdots \omega_{i_{d}}\tilde{\varphi}\left(  X^{n},\sqrt{\Delta t}P_{Q},\Delta
tQ\right)  ,\nonumber
\end{align}
where\ $\tilde{\varphi}\left(  x,y,z\right)  =\varphi \left(  x+b(t_{n}%
,x)\Delta t+\sigma(t_{n},x)y+\langle h(t_{n},x),z\rangle,y,z\right)  $, the
weights $\omega_{i}=\frac{2^{L+1}L!}{\left[  H_{L}^{\prime}\left(
x_{i}\right)  \right]  ^{2}}$ for $i=1,\ldots L$, $P_{Q}\allowbreak=\sqrt
{Q}(\sqrt{2}p_{i_{1}},\ldots,\sqrt{2}p_{i_{d}})^{\top}$, and the nodes
$\{p_{i}\}_{i=1}^{L}\ $are the roots of the $L$ degree Hermite polynomial
$H_{L}\left(  x\right)  =(-1)^{L}e^{x^{2}}\frac{d^{L}}{dx^{L}}e^{-x^{2}}$.
Similarly, we define the associated discrete sublinear expectation
\begin{equation}
\mathbb{\tilde{E}}^{GH}\left[  \varphi_{t_{n+1}}\right]  =\mathbb{\tilde{E}%
}_{t_{0}}^{GH,x_{0}}\left[  \mathbb{\tilde{E}}_{t_{1}}^{GH,X^{1}}\left[
\cdots \mathbb{\tilde{E}}_{t_{n}}^{GH,X^{n}}[\varphi_{t_{n+1}}]\right]
\right]  .
\end{equation}

\begin{remark}
\label{remark1}For any fixed $L\geq2$, noting that the Gauss-Hermite
quadrature rule is exact for any polynomial of degree $\leq2L-1$, one can
check that
\[
\sum \limits_{i=1}^{L}\omega_{i}=1,\text{ \ }\sum \limits_{i=1}^{L}\omega
_{i}p_{i}=0,\text{ \ }\sum \limits_{i=1}^{L}\omega_{i}p_{i}^{2}=\frac{1}{2}.
\]

\end{remark}

Similar to the one-dimensional case, based on (\ref{7.1}) and (\ref{7.2}),
we get the following reference equations:
\begin{align}
&  Y_{t_{n}}^{t_{n},X^{n}}=\mathbb{\tilde{E}}_{t_{n}}^{GH,X^{n}}\left[
Y_{t_{n+1}}^{t_{n+1},X^{n+1}}+f_{t_{n+1}}^{t_{n+1},X^{n+1}}\Delta t+\langle
g_{t_{n+1}}^{t_{n+1},X^{n+1}},\Delta \langle B\rangle_{n+1}\rangle \right]
+\tilde{R}_{y}^{n}+R_{y}^{G,n},\label{Y}\\
&  Z_{t_{n}}^{t_{n},X^{n}}Q\Delta t=\mathbb{E}_{t_{n}}^{GH,Q,X^{n}}\left[
Y_{t_{n+1}}^{t_{n+1},X^{n+1}}\Delta B_{n+1}^{\top}\right]  +\tilde{R}%
_{z}^{n,Q}+R_{z}^{G,n,Q},\text{ for }Q\in \Sigma, \label{Z}%
\end{align}
where
\begin{align*}
&  R_{y}^{G,n}=\mathbb{\tilde{E}}_{t_{n}}^{X^{n}}\left[  Y_{t_{n+1}}%
^{t_{n+1},X^{n+1}}+f_{t_{n+1}}^{t_{n+1},X^{n+1}}\Delta t+\langle g_{t_{n+1}%
}^{t_{n+1},X^{n+1}},\Delta \langle B\rangle_{n+1}\rangle \right]\\
&  \text{ \  \  \  \  \  \  \  \ }-\mathbb{\tilde{E}}_{t_{n}}^{GH,X^{n}}\left[
Y_{t_{n+1}}^{t_{n+1},X^{n+1}}+f_{t_{n+1}}^{t_{n+1},X^{n+1}}\Delta t+\langle
g_{t_{n+1}}^{t_{n+1},X^{n+1}},\Delta \langle B\rangle_{n+1}\rangle \right]
, \\
&  R_{z}^{G,n,Q}=\mathbb{E}_{t_{n}}^{Q,X^{n}}\left[  Y_{t_{n+1}}%
^{t_{n+1},X^{n+1}}\Delta B_{n+1}^{\top}\right]  -\mathbb{E}_{t_{n}%
}^{GH,Q,X^{n}}\left[  Y_{t_{n+1}}^{t_{n+1},X^{n+1}}\Delta B_{n+1}^{\top
}\right]  .
\end{align*}

\subsection{Numerical schemes and convergence results}

Let $Y^{n}$ and $Z^{n}$ be the approximate values of the solutions $Y_{t}$ and
$Z_{t}$ of the $G$-FBSDE (\ref{1.0}) at time $t_{n}$, respectively, and denote
$f^{n+1}=f(t_{n+1},X^{n+1},Y^{n+1},Z^{n+1})$, $g^{n+1}=g(t_{n+1},X^{n+1}$,
$Y^{n+1},Z^{n+1})$. By removing the error terms $\tilde{R}_{y} ^{n}%
,R_{y}^{G,n},\tilde{R}_{z}^{n,Q}$, and $R_{z}^{G,n,Q}$ from (\ref{Y}%
)-(\ref{Z}), we obtain our discrete scheme for solving the $G$-FBSDE
(\ref{1.0}) as follows.

\begin{sch}
\label{GH scheme}Given random variables $Y^{N}$ and $Z^{N}$, for
$n=N-1,\ldots,0$ and $Q$ $\in \Sigma$, solve random variables $Y^{n}%
=Y^{n}(X^{n})$ and $Z^{n}=Z^{n}(X^{n})$ from%
\begin{align}
Y^{n}  &  =\mathbb{\tilde{E}}_{t_{n}}^{GH,X^{n}}\left[  Y^{n+1}+f^{n+1}\Delta
t+\langle g^{n+1},\Delta \langle B\rangle_{n+1}\rangle \right]  ,\label{yn_GH}\\
Z^{n}  &  =\mathbb{E}_{t_{n}}^{GH,Q,X^{n}}\left[  Y^{n+1}\Delta B_{n+1}^{\top
}\right]  Q^{-1}/\Delta t, \label{zn_GH}%
\end{align}
with%
\begin{equation}
X^{n+1}=X^{n}+b(t_{n},X^{n})\Delta t+\langle h(t_{n},X^{n}),\Delta \langle
B\rangle_{n+1}\rangle+\sigma(t_{n},X^{n})\Delta B_{n+1}.
\end{equation}
\
\end{sch}

Compared with the trinomial tree scheme in the one-dimensional $G$-Brownian motion case,
(\ref{zn_GH}) in Scheme \ref{GH scheme} for $Z^{n}$ depends on the parameter
$Q$ in the multi-dimensional case. To obtain the convergence results for
Scheme \ref{GH scheme}, the following assumption need to be imposed.

\begin{enumerate}
\item[(A3)] There exists a constant $C>0$, such that, for any $\tilde{Q}%
\in \Sigma$,
\begin{equation}
\sup \limits_{0\leq n\leq N-1}|Z^{n}-\tilde{Z}^{n}|\leq C\left(  \Delta
t\right)  ^{\frac{1}{2}}, \label{A3}%
\end{equation}
where
\[
\tilde{Z}^{n}=\mathbb{E}_{t_{n}}^{GH,\tilde{Q},X^{n}}\left[  Y^{n+1}\Delta
B_{n+1}\right]  \tilde{Q}^{-1}/\Delta t.
\]

\end{enumerate}

Let us give a sufficient condition for the assumption (A3) to illustrate its rationality.

\begin{proposition}
For each fixed $(t_{n},X^{n})$, assume that there exists a function $\psi
_{n}\in C^{1}(\mathbb{R}^{d}\times \mathbb{R}^{d\times d};\mathbb{R)}\ $such
that $Y^{n+1}=Y^{n+1}(X^{n+1}) =\psi_{n}(\Delta B_{n+1},\Delta \langle
B\rangle_{n+1})$, and for any $x,x^{\prime}\in \mathbb{R}^{d}$, $y,y^{\prime
}\in \mathbb{R}^{d\times d}$,
\begin{equation}%
\begin{array}
[c]{r}%
\sup \limits_{0\leq n\leq N-1}\left \vert D_{x}\psi_{n}\left(  x,y\right)
-D_{x}\psi_{n}\left(  x^{\prime},y^{\prime}\right)  \right \vert \leq C\left(
\left \vert x-x^{\prime}\right \vert +\left \vert y-y^{\prime}\right \vert
\right)  ,\\
\sup \limits_{0\leq n\leq N-1}\left \vert D_{y}\psi_{n}\left(  x,y\right)
-D_{y}\psi_{n}\left(  x^{\prime},y^{\prime}\right)  \right \vert \leq C\left(
\left \vert x-x^{\prime}\right \vert +\left \vert y-y^{\prime}\right \vert
\right)  .
\end{array}
\label{4.9.0}%
\end{equation}
Then, the assumption \emph{(A3)} holds.
\end{proposition}

\begin{proof}
Notice that
\begin{equation}
\mathbb{E}_{t_{n}}^{GH,Q,X^{n}}\left[  Y^{n+1}\Delta B_{n+1}^{\top}\right]
=\sum_{i_{1},\ldots,i_{d}=1}^{L}\omega_{i_{1}}\cdots \omega_{i_{d}}\psi
_{n}(\sqrt{\Delta t}P_{Q},\Delta tQ)\sqrt{\Delta t}P_{Q}^{\top}. \label{4.9.1}%
\end{equation}
Using the Taylor expansion and (\ref{4.9.0}), for $\theta \in(0,1)$, we have
\begin{align}
\psi_{n}(\sqrt{\Delta t}P_{Q},\Delta tQ)  &  =\psi_{n}(0,0)+D_{x}\psi
_{n}(0,0)P_{Q}\sqrt{\Delta t}+\langle D_{y}\psi_{n}(0,0),Q\rangle \Delta
t\text{ \  \  \ }\label{4.9.2}\\
&  \text{ \  \  \ }+\int_{0}^{1}[D_{x}\psi_{n}(\theta \sqrt{\Delta t}P_{Q}%
,\theta \Delta tQ)-D_{x}\psi_{n}(0,0)]d\theta P_{Q}\sqrt{\Delta t}\nonumber \\
&  \text{ \  \  \ }+\int_{0}^{1}\langle D_{y}\psi_{n}(\theta \sqrt{\Delta t}%
P_{Q},\theta \Delta tQ)-D_{y}\psi_{n}(0,0),Q\rangle d\theta \Delta t\nonumber \\
&  =\psi_{n}(0,0)+D_{x}\psi_{n}(0,0)P_{Q}\sqrt{\Delta t}+\mathcal{O}(\Delta
t).\nonumber
\end{align}
From (\ref{4.9.1})-(\ref{4.9.2}), one can check that%
\begin{equation}
Z^{n}=D_{x}\psi_{n}(0,0)+\mathcal{O}(\Delta t)^{1/2}.
\end{equation}
Analogously,
\begin{equation}
\tilde{Z}^{n}=D_{x}\psi_{n}(0,0)+\mathcal{O}(\Delta t)^{1/2}.
\end{equation}
The desired result follows.
\end{proof}

Similar to the analysis in Theorems \ref{theorem 1} and \ref{theorem 2}, under
the assumptions (A1)-(A3), we derive the following convergence theorem.

\begin{theorem}
\label{theorem 3}Suppose \emph{(A1)}-\emph{(A3)} hold. Let $(X_{t}
^{t_{n},X^{n}},Y_{t}^{t_{n},X^{n}},Z_{t}^{t_{n},X^{n}})_{t_{n}\leq t\leq T}$
be the solution of \eqref{G-SDE}-\eqref{G-BSDE}, and $(X^{n},Y^{n}%
,Z^{n})(n=0,1,\ldots,N)$ be the solution to Scheme \ref{GH scheme} with
$Y^{N}=u(t_{N},X^{N})$ and $Z^{N} =D_{x}u(t_{N},X^{N}) \sigma(t_{N},X^{N})$.
Then, for sufficiently small time step $\Delta t$,
\[
\mathbb{\tilde{E}}^{GH}\left[  |Y_{t_{n}}^{t_{n},X^{n}}-Y^{n}|^{2}+\Delta
t|Z_{t_{n}}^{t_{n},X^{n}}-Z^{n}|^{2}\right]  \leq C\Delta t.
\]

\end{theorem}

\begin{remark}
In the case when $f=g=0$, without the assumption $(A3)$, we have
\[
\mathbb{\tilde{E}}^{GH}\left[  |Y_{t_{n}}^{t_{n},X^{n}}-Y^{n}|^{2}\right]
\leq C\Delta t,
\]
which indicates that Scheme \ref{GH scheme} is of order $1/2$ for solving the
$G$-heat equation introduced by Peng \cite{P2010}
\begin{equation}
\left \{
\begin{array}
[c]{l}%
\partial_{t}u+G\left(  D_{xx}^{2}u\right)  =0,\text{\  \ }\forall \left(  t,x\right)
\in(0,T]\times \mathbb{R}^{m},\\
u\left(  T,x\right)  =\phi \left(  x\right)  .
\end{array}
\right.  \label{G-heat}%
\end{equation}

\end{remark}

\section{Numerical experiments}

In this section, some numerical experiments will be carried out to illustrate
the high accuracy of our numerical schemes. We take a uniform partition with
the time step $\Delta t=\frac{T}{N}$ and introduce the uniform space
partition $\mathcal{D}_{h}=\mathcal{D}_{1,h}\times \mathcal{D}_{2,h}%
\times \cdots \times \mathcal{D}_{m,h}$, where $\mathcal{D}_{j,h}$ is the
partition of the one-dimensional real axis $\mathbb{R}$
\[
\mathcal{D}_{j,h}=\left \{  \left.  x_{k}^{j}\right \vert x_{k}^{j}=k\Delta
x,\text{ }k=0,\pm1,\pm2,\ldots \right \}  ,
\]
for $j=1,2,\ldots,m$, and $\Delta x$ is a suitable spatial step. Let
$|Y_{0}-Y^{0}|$ and $|Z_{0}-Z^{0}|$ represent the errors between the exact and
numerical solutions for $Y$ and $Z$ at $(t_{0},x_{0})$, and denote
$\varphi(x)=x\wedge1\vee(-1)$. In our tables, we also denote by CR the
convergence rate, TR the discrete Scheme \ref{trinomial tree scheme}, and GH
the discrete Scheme \ref{GH scheme}, respectively. When it comes to Scheme
\ref{GH scheme}, in order to avoid confusion, we assume that $\mathbb{\tilde
{E}}_{t_{n}}^{GH,X^{n}}[\cdot]$ reaches its maximum at $\mathbb{E}_{t_{n}%
}^{GH,Q_{y}^{n},X^{n}}[\cdot]$ with the parameter $Q_{y}^{n}$, and use the
notation $Q_{z}^{n}$ to further distinguish the selected parameter for solving
$Z$.

\subsection{$G$-heat equation}

We first apply our schemes to the $G$-heat equation \eqref{G-heat} which is
related to the\ $G$-FBSDE \eqref{G-SDE}-\eqref{G-BSDE} with $f=g=0$.

\begin{example}
\label{EX4}Consider the $G$-heat equation \eqref{G-heat} with the terminal
$\varphi(x)=(x+c_{1})^{3}$. We set $T=1$, $x_{0}=0$, $\underline{\sigma}=0.2$,
$\overline{\sigma}=1$, and $c_{1}=-0.5840$. The solution of this
$G$-heat equation given in \cite{H2012} is $u(t,x)=(1-t)^{\frac{3}{2}}P\left(
\frac{x+c_{1}}{\sqrt{1-t}}\right)$, where
\[
P\left(  x\right)  =\left \{
\begin{array}
[c]{ll}%
\displaystyle3x+x^{3}+\frac{k_{1}}{2}\left[  \left(  2+x^{2}\right)
e^{-\frac{x^{2}}{2}}-\left(  3x+x^{3}\right)  \int_{x}^{\infty}e^{-\frac
{r^{2}}{2}}dr\right]  , & x\geq \overline{c}_{1},\\
\displaystyle3\underline{\sigma}^{2}x+x^{3}+\frac{d_{1}}{2\underline{\sigma
}^{2}}\left[  \left(  2\underline{\sigma}^{2}+x^{2}\right)  e^{-\frac{x^{2}
}{2\underline{\sigma}^{2}}}+\frac{1}{\underline{\sigma}}\left(  3\underline
{\sigma}^{2}x+x^{3}\right)  \int_{-\infty}^{\frac{x}{\underline{\sigma}}%
}e^{-\frac{r^{2}}{2}}dr\right]  , & x<\overline{c}_{1},
\end{array}
\right.
\]
with $\overline{c}_{1}=-0.5840$, $k_{1}=0.6154$ and $d_{1}=36.8406$. The exact solution at $(t_0,x_0)$ is $(u(0,0),\partial
_{x} u(0,0))=(Y_{0} ,Z_{0})=(-0.2595,1.3331)$. We solve this example by Scheme
\ref{trinomial tree scheme} and Scheme \ref{GH scheme} with $Q_{z}^{n}
=Q_{y}^{n}$. The numerical errors and the corresponding convergence rates are
listed in Table \ref{Table 4.1}, which are consistent with our theoretical
results. \begin{table}[tbh]
\caption{Errors and convergence rates for Example \ref{EX4} with
$c_{1}=-0.584$.}%
\label{Table 4.1}%
{\footnotesize
\[%
\begin{tabular}
[c]{||c|c|c|c|c|c|c||}\hline
\multicolumn{7}{||c||}{$\left \vert Y_{0}-Y^{0}\right \vert $}\\ \hline
Scheme & $N=16$ & $N=32$ & $N=64$ & $N=128$ & $N=256$ & CR\\ \hline
TR & 9.347E-04 & 6.110E-04 & 4.286E-04 & 2.886E-04 & 1.764E-04 & 0.589\\ \hline
GH & 7.005E-03 & 4.966E-03 & 3.259E-03 & 2.002E-03 & 1.158E-03 & 0.650\\ \hline
\multicolumn{7}{||c||}{$\left \vert Z_{0}-Z^{0}\right \vert $}\\ \hline
Scheme & $N=16$ & $N=32$ & $N=64$ & $N=128$ & $N=256$ & CR\\ \hline
TR & 1.216E-01 & 7.390E-02 & 4.183E-02 & 2.203E-02 & 1.088E-02 & 0.871\\ \hline
GH & 4.918E-02 & 3.377E-02 & 2.320E-02 & 1.520E-02 & 9.348E-03 & 0.594\\ \hline
\end{tabular}
\]
}\end{table}

Next, we choose a different $c_{1}=0$. The unique solution of the
corresponding $G$-heat equation is $u(t,x)=(1-t)^{\frac{3}{2}}P\left(
\frac{x}{\sqrt{1-t}}\right)$, and
the exact solution is $(u(0,0),\partial_{x} u(0,0))=(Y_{0},Z_{0}%
)=(0.6154,1.8430)$. We solve this example by Scheme
\ref{trinomial tree scheme} and Scheme \ref{GH scheme} with $Q_{z}^{n}
=Q_{y}^{n}$ again, and the errors and the corresponding convergence rates
are given in Table \ref{Table 4.2}. \begin{table}[tbh]
\caption{Errors and convergence rates for Example \ref{EX4} with $c_{1}=0$.}%
\label{Table 4.2}%
{\footnotesize
\[%
\begin{tabular}
[c]{||c|c|c|c|c|c|c||}\hline
\multicolumn{7}{||c||}{$\left \vert Y_{0}-Y^{0}\right \vert $}\\ \hline
Scheme & $N=16$ & $N=32$ & $N=64$ & $N=128$ & $N=256$ & CR\\ \hline
TR & 9.097E-03 & 4.294E-03 & 2.564E-03 & 1.123E-03 & 4.811E-04 & 1.042\\ \hline
GH & 1.455E-02 & 7.413E-03 & 3.965E-03 & 1.900E-03 & 9.534E-04 & 0.983\\ \hline
\multicolumn{7}{||c||}{$\left \vert Z_{0}-Z^{0}\right \vert $}\\ \hline
Scheme & $N=16$ & $N=32$ & $N=64$ & $N=128$ & $N=256$ & CR\\ \hline
TR & 7.694E-02 & 3.975E-02 & 1.843E-02 & 9.404E-03 & 5.028E-03 & 0.995\\ \hline
GH & 1.328E-02 & 6.853E-03 & 5.800E-03 & 2.434E-03 & 1.072E-03 & 0.876\\ \hline
\end{tabular}
\]
}\end{table}

By contrast, we observed that when $c_{1}=0$, the convergence rates of $\vert
Y_{0}-Y^{0} \vert$ and $\vert Z_{0}-Z^{0}\vert$ are higher than that when
$c_{1}=-0.5840$. This is due to the fact that $x_{0}=0$ is the inflection
point of the solution $u(0,x)$ in the case when $c_{1}=-0.5840$, but $x_{0}=0$
is not the inflection point of $u(0,x)$ when $c_{1}=0$. At the convex or
concave point, the calculation of $G$-Brownian motion distribution degenerates
to that of classical Brownian motion distribution, which leads to a higher
convergence rate than the theoretical result. However, the conclusion does not
hold at the inflection point.
\end{example}

\subsection{$G$-FBSDEs}

In the following, we will apply our schemes to solve $G$-FBSDEs. The first
example is the case of one-dimensional $G$-Brownian motion, and the second is
the multi-dimensional case.

\begin{example}
\label{EX2}We consider the following $G$-FBSDEs:
\begin{equation}
\left \{
\begin{aligned} X_{t} & =x_{0}+\int_{0}^{t}\frac{1}{1+2\exp ( s+X_{s}) }ds+\int_{0}^{t}\frac{\exp ( s+X_{s}) }{1+\exp ( s+X_{s}) }dB_{s},\text{ \  \ }0\leq t\leq T,\\ Y_{t} & =\frac{\exp ( T+X_{T}) }{1+\exp ( T+X_{T}) }-\int_{t}^{T}\left[ \frac{Y_{s}}{1+2\exp ( s+X_{s}) }+G\left( \varphi( 2Y_{s}^{2})-1\right) \right] ds\\ & -\frac{1}{2}\int_{t}^{T}\left[ 1+\frac{\varphi(Y_{s})\varphi(Z_{s})}{1+\exp( s+X_{s}) }-\varphi(Y_{s}^{2})\left(2+\varphi(Z_s)\right) \right] d\langle B\rangle_{s}-\int_{t}^{T}Z_{s}dB_{s}-(K_{T}-K_{t}) . \end{aligned}\right.
\label{ex2}%
\end{equation}
The analytic solution of \eqref{ex2} is
\[
\left \{
\begin{aligned} \displaystyle Y_{t} & =\frac{\exp ( t+X_{t}) }{1+\exp ( t+X_{t}) },\text{ \  \ }Z_{t}=\frac{( \exp ( t+X_{t}) ) ^{2}}{( 1+\exp ( t+X_{t}) ) ^{3}},\\ \displaystyle K_{t} & =\int_{0}^{t}\left[ \frac{( \exp ( s+X_{s}) ) ^{2}}{( 1+\exp ( s+X_{s})) ^{2}}-\frac{1}{2}\right] d\langle B\rangle_{s}-\int_{0}^{t}G\left( \frac{2( \exp ( s+X_{s}) ) ^{2}}{( 1+\exp ( s+X_{s}) ) ^{2}}-1\right) ds. \end{aligned}\right.
\]

Set $T=1$, $x_{0}=1$, $\underline{\sigma}=0.7$, and $\overline{\sigma}=1$. The
exact solution is $(Y_{0},Z_{0})=(0.731,0.144)$. We test this example by
Scheme \ref{trinomial tree scheme} and Scheme \ref{GH scheme} with different
$Q_{z}^{n}$. The numerical errors and the corresponding convergence rates for
$N=16,32,64,128,256$ are listed in Table \ref{Table 2}.
It is clear that the selection of $Q_{z}^{n}$ in Scheme \ref{GH scheme} does
not affect the convergence rate, which is consistent with our theoretical
results. \begin{table}[ptbh]
\caption{Errors and convergence rates for Example \ref{EX2}.}%
\label{Table 2}%
{\footnotesize
\[%
\begin{tabular}
[c]{||c|l|c|c|c|c|c|c||}\hline
\multicolumn{8}{||c||}{$\left \vert Y_{0}-Y^{0}\right \vert $}\\ \hline
\multicolumn{2}{||c|}{Scheme} & $N=16$ & $N=32$ & $N=64$ & $N=128$ & $N=256$ &
CR\\ \hline
\multicolumn{2}{||c|}{TR} & 1.656E-03 & 7.642E-04 & 3.746E-04 & 1.808E-04 &
8.960E-05 & 1.050\\ \hline
& $Q_{z}^{n}=Q_{y}^{n}$ & 1.890E-03 & 8.803E-04 & 4.246E-04 & 2.056E-04 &
1.016E-04 & 1.053\\ \cline{2-8}%
GH & $Q_{z}^{n}=\underline{\sigma}$ & 1.832E-03 & 8.519E-04 & 4.105E-04 &
1.985E-04 & 9.806E-04 & 1.055\\ \cline{2-8}
& $Q_{z}^{n}=\overline{\sigma}$ & 1.888E-03 & 8.796E-04 & 4.243E-04 &
2.054E-04 & 1.015E-04 & 1.053\\ \hline
\multicolumn{8}{||c||}{$\left \vert Z_{0}-Z^{0}\right \vert $}\\ \hline
\multicolumn{2}{||c|}{Scheme} & $N=16$ & $N=32$ & $N=64$ & $N=128$ & $N=256$ &
CR\\ \hline
\multicolumn{2}{||c|}{TR} & 5.004E-03 & 2.152E-03 & 9.975E-04 & 4.908E-04 &
2.657E-04 & 1.060\\ \hline
& $Q_{z}^{n}=Q_{y}^{n}$ & 5.731E-03 & 2.625E-03 & 1.213E-03 & 5.666E-04 &
2.745E-04 & 1.098\\ \hline
GH & $Q_{z}^{n}=\underline{\sigma}$ & 5.381E-03 & 2.398E-03 & 1.105E-03 &
5.308E-04 & 2.731E-04 & 1.078\\ \hline
& $Q_{z}^{n}=\overline{\sigma}$ & 5.728E-03 & 2.624E-03 & 1.212E-03 &
5.661E-04 & 2.743E-04 & 1.098\\ \hline
\end{tabular}
\]
}\end{table}
\end{example}

\begin{example}
\label{EX5}In this example, let $G(\cdot)$ be a sublinear function defined in \eqref{G}
with $\Sigma=\{ \lambda Q_{1}+(1-\lambda)Q_{2}, \lambda \in \lbrack0,1]\}$, where
\[
Q_{1}=\left(
\begin{array}
[c]{cc}%
2 & 1\\
1 & 1
\end{array}
\right)  \text{, \  \  \ }Q_{2}=\left(
\begin{array}
[c]{cc}%
1 & 1\\
1 & 2
\end{array}
\right)  .
\]
Let $B_{t}=(B_{t}^{1},B_{t}^{2})^{\top}$ be a two-dimensional $G$-Brownian motion.
We apply Scheme \ref{GH scheme} to the following $G$-FBSDEs:
\begin{equation}
\left \{
\begin{array}
[c]{l}%
\displaystyle Y_{t}=Y_{T}-\int_{t}^{T}\left(  Z_{s}^{1}+Z_{s}^{2}+G\left(
M_{s}\right)  \right)  ds-\int_{t}^{T}Z_{s}dB_{s}-\left(  K_{T}-K_{t}\right)
,\\
\displaystyle Y_{T}=\sin(T+B_{T}^{1})\cos(T+B_{T}^{2}),
\end{array}
\right.  \label{ex5}%
\end{equation}
where $Z_{s}=(Z_{s}^{1},Z_{s}^{2})$, $M_{s}=\left(
\begin{array}
[c]{cc}%
-Y_{s} & V_{s}\\
V_{s} & -Y_{s}%
\end{array}
\right)  $, and $V_{s}=-\cos \left(  s+B_{s}^{1}\right)  \sin \left(
s+B_{s}^{2}\right)  $. The analytic solution of \eqref{ex5}
is given by
\[
\left \{
\begin{array}
[c]{l}%
Y_{t}=\sin \left(  t+B_{t}^{1}\right)  \cos \left(  t+B_{t}^{2}\right)  ,\\
Z_{t}=\left(  \cos(t+B_{t}^{1})\cos \left(  t+B_{t}^{2}\right)  , -\sin \left(
t+B_{t}^{1}\right)  \sin \left(  t+B_{t}^{2}\right)  \right)  ,\\
K_{t}=\frac{1}{2}\int_{0}^{t}M_{s}d\langle B\rangle_{s}-\int_{0}^{t}%
G(M_{s})ds.
\end{array}
\right.
\]

Set $T=1$, $x_{0}=1$, and $L=6$. The exact solution is $(Y_{0},Z_{0})=(0,(1,0))$.
In our tests, we found that the difference between choosing different $Q_{z}
^{n}\in \{Q_{1},Q_{2}\}$ is negligible. We only list
the numerical errors and the convergence rates for $N=16,32,64,128,256$ with
$Q_{z}^{n}=Q_{1}$ in Table \ref{Table 5}. It is shown that our method is stable and has high
accuracy. \begin{table}[ptbh]
\caption{Errors and convergence rates of Scheme \ref{GH scheme} for Example
\ref{EX5}.}%
\label{Table 5}
{\footnotesize
\[%
\begin{tabular}
[c]{||c|c|c|c|c|c|c||}\hline
\multicolumn{7}{||c||}{$\left \vert Y_{0}-Y^{0}\right \vert $}\\ \hline
GH & $N=16$ & $N=32$ & $N=64$ & $N=128$ & $N=256$ & CR\\ \hline
$Q_{z}^{n}=Q_{1}$ & 1.838E-01 & 9.928E-02 & 5.015E-02 & 2.497E-02 &
1.243E-02 & 0.976\\ \hline
\multicolumn{7}{||c||}{$\left \vert Z_{0}-Z^{0}\right \vert $}\\ \hline
GH & $N=16$ & $N=32$ & $N=64$ & $N=128$ & $N=256$ & CR\\ \hline
$Q_{z}^{n}=Q_{1}$ & 1.631E-01 & 5.685E-02 & 1.980E-02 & 7.609E-03 &
3.251E-03 & 1.420\\ \hline
\end{tabular}
\  \  \
\]
}\end{table}
\end{example}

\section{Conclusions}

In this paper, we propose some efficient numerical schemes for solving
$G$-FBSDEs. With the help of the $G$-expectation representation, we design a
feasible method to approximate the conditional $G$-expectation. Then by using
the trinomial tree rule and the Gauss-Hermite quadrature rule to approximate
the distribution of $G$-Brownian motion, we propose some new numerical schemes
for solving $G$-FBSDEs. We also rigorously analyze the errors of our proposed
schemes and prove the convergence results. Several numerical examples are
presented to show the effectiveness of our numerical schemes.

\end{document}